\documentclass[british]{article}[11pt]
\pdfoutput=1%
\usepackage[T1]{fontenc}%
\usepackage[utf8]{inputenc}%
\usepackage{babel}%
\usepackage{csquotes}%
\usepackage{amssymb}%
\usepackage{mathtools}%
\usepackage{amsthm}%
\usepackage{microtype}%
\usepackage{enumitem}%
\renewcommand{\epsilon}{\varepsilon}

\usepackage{verbatim}%

\usepackage[a4paper,
            left=1.4in,
            right=1.4in,
            top=1.3in,
            bottom=1.3in]{geometry}
            
\usepackage[backend=biber,doi=false,url=false,isbn=false,date=year,
giveninits=true,maxbibnames=99,sortcites,hyperref=true,backref=true]{biblatex}
\renewbibmacro{in:}{%
  \ifentrytype{article}{}{\printtext{\bibstring{in}\intitlepunct}}}
\DeclareFieldFormat[article,periodical]{volume}{\mkbibbold{#1}}%
\DeclareFieldFormat[article,periodical,unpublished]{citetitle}{#1}
\DeclareFieldFormat[article,periodical,unpublished]{title}{#1}
\DeclareFieldFormat{pages}{#1}
\DeclareFieldFormat[inproceedings]{title}{#1\isdot}
\DeclareFieldFormat[misc]{title}{#1\isdot}
\DeclareFieldFormat[misc,article]{note}{#1\addcomma}

\AtEveryBibitem{
  \clearfield{number}}
\usepackage[pdfencoding=auto,%
psdextra%
]{hyperref}%

\theoremstyle{plain}
\newtheorem{thm}{Theorem}[section]
\newtheorem{lemma}[thm]{Lemma}
\newtheorem{cor}[thm]{Corollary}
\newtheorem{prop}[thm]{Proposition}

\theoremstyle{definition}
\newtheorem{defn}[thm]{Definition}
\newtheorem{example}[thm]{Example}

\numberwithin{equation}{section}

\begin{document}

\title{Generalised intermediate dimensions}
\date{}

\author{Amlan Banaji\footnote{
\textsc{Department of Mathematical Sciences, Loughborough University, Loughborough, LE11 3TU, United Kingdom} \\
 \textit{Email:} \texttt{A.F.Banaji@lboro.ac.uk}}}

\maketitle

\begin{abstract}
We introduce a family of dimensions, which we call the $\Phi$-intermediate dimensions, that lie between the Hausdorff and box dimensions and generalise the intermediate dimensions introduced by Falconer, Fraser and Kempton. This is done by restricting the relative sizes of the covering sets in a way that allows for greater refinement than in the definition of the intermediate dimensions. 
We also extend the theory from Euclidean space to a wider class of metric spaces. We prove that these dimensions can be used to `recover the interpolation' between the Hausdorff and box dimensions of compact subsets for which the intermediate dimensions are discontinuous at $\theta=0$, thus providing finer geometric information about such sets. 
We prove continuity-like results involving the Assouad and lower dimensions, which give a sharp general lower bound for the intermediate dimensions that is positive for all $\theta \in (0,1]$ for sets with positive box dimension. We also prove H\"older distortion estimates, a mass distribution principle, and a Frostman type lemma, which we use to study dimensions of product sets. 
\vspace{0.2cm}


\end{abstract}

\vspace{-0.5cm}

\tableofcontents

\section{Introduction} 

When studying the geometry of fractal subsets of a metric space, it is common to consider different notions of dimension, which attempt to quantify the extent to which the set fills up space at small scales. 
Two of the most familiar are the Hausdorff and box dimensions. 
For many natural sets these differ, indicating (intuitively) that the set in question has some inhomogeneity. The key difference between these two dimensions it that in the definition of box dimension the covering sets are required to be of equal size, but for Hausdorff dimension there is no such restriction. 
In~\cite{Falconer2020}, Falconer, Fraser and Kempton introduced a family of dimensions, called the \emph{intermediate dimensions}, which depend on a parameter $\theta \in [0,1]$, by insisting that the sizes of the covering sets lie in intervals of the form $[\delta^{1/\theta},\delta]$. The Hausdorff and box dimensions are the two extreme cases $\theta = 0$ and $1$, respectively. The intermediate dimensions have been studied in~\cite{Banaji2021moran,Falconer2021,
Tan2020,Burrell2022spirals,Falconer2021-2,Daw2021} and other works. 
For classes of fractal sets such as Bedford--McMullen carpets~\cite{Banaji2021bedford} and infinitely generated self-conformal sets~\cite{Banaji2021}, they have been computed explicitly. 

For every set, the intermediate dimensions are continuous at each $\theta \in (0,1]$. For many sets, such as Bedford--McMullen carpets (see~\cite[Section~4]{Falconer2020} and \cite{Banaji2021bedford}) and polynomial sequences (see~\cite[Proposition~3.1]{Falconer2020}), they are also continuous at $\theta = 0$, so fully interpolate between the Hausdorff and box dimensions. %
Continuity of the intermediate dimensions at $\theta = 0$ has powerful consequences, in particular for the box dimensions of projections of the set~\cite{Burrell2021} and images of the set under stochastic processes such as fractional Brownian motion~\cite{Burrell2022brownian}. 
However, for many sets the intermediate dimensions are discontinuous at $\theta=0$, or even constant at the value of the box dimension, in which case they give very little information about the set. One such set which is also compact is $\{0\} \cup \{\,\frac{1}{\log n} : n \in \mathbb{N},n \geq 3 \,\}$. In this paper, we introduce the $\Phi$-intermediate dimensions, by restricting the sizes of the covering sets to lie in a wider class of intervals of the form $[\Phi(\delta),\delta]$ for more general functions $\Phi$. These dimensions give even more refined geometric information than the intermediate dimensions about sets for which the intermediate dimensions are discontinuous at $\theta = 0$. %
Indeed, in what is perhaps the most important result of this paper (Theorem~\ref{recoverinterpolation}), we demonstrate that if a set is compact then there will always be a family of functions~$\Phi$ which interpolate all the way between the Hausdorff and box dimensions. 
While many results for the $\Phi$-intermediate dimensions are similar to results for the intermediate dimensions, others, such as the H{\"o}lder distortion estimates in Theorem~\ref{holder}, are rather different. 
We believe that the results of this paper demonstrate that the $\Phi$-intermediate dimensions give rise to a rich and workable theory in their own right. 
It is natural to ask whether the potential-theoretic methods in \cite{Burrell2021,Burrell2022brownian} can be adapted to study the $\Phi$-intermediate dimensions.  Feng~\cite{FengArxivintermediate} has recently shown that this is indeed the case, obtaining information about dimensions of images of sets under projections and fractional Brownian motion if for all $\epsilon > 0$ the function $\Phi$ satisfies $\delta^{\epsilon} \log \Phi(\delta) \to 0$ as $\delta \to 0$. 

The intermediate dimensions are an example of \emph{dimension interpolation}, an area which was introduced relatively recently but has gathered significant interest. For a survey of this topic we refer the reader to~\cite{Fraser2021-1}. The idea is to consider two different notions of dimension and find a geometrically meaningful family of dimensions which lie between them and share some characteristics of both, but provide more information about sets than either does in isolation. The hope is that, as well as being interesting in its own right, dimension interpolation can help illuminate why for some sets the two endpoint dimensions can give different values. A different example of dimension interpolation is the Assouad spectrum, introduced by Fraser and Yu in~\cite{Fraser2018-2}, which lies between the upper box and Assouad dimensions, giving information about the `thickest' part of the set. A more general class of dimensions were also introduced in~\cite{Fraser2018-2}, greatly developed by Garc{\'{i}}a, Hare and Mendivil in~\cite{Garcia2020}, and further studied in~\cite{Garcia2019-2,Troscheit2019,Hare2022randomass,
HarePreprintrandomass2}. 
They are defined by fixing the relative scales in more general ways than for the Assouad spectrum, thus giving more refined geometric information about sets whose quasi-Assouad dimension is less that the Assouad dimension. These Assouad-like dimensions were part of our original motivation for considering the $\Phi$-intermediate dimensions. 

\subsection{Summary and discussion of main results}
In Section~\ref{prelimsection}, we introduce the notation and the types of metric spaces that we work with, and make some standing assumptions to reduce repetition. We also define the notions of dimension that we will need. 

In Section~\ref{ctysection}, we give relationships between the different notions of dimension (Propositions~\ref{basicbounds} and~\ref{compareintermediate}). In Theorem~\ref{maincty} and Proposition~\ref{zerocontinuityprop} we prove quantitative %
continuity-like properties for the $\Phi$-intermediate dimensions, which intuitively say that if two functions $\Phi$ and $\Phi_1$ are `close' to each other then the dimensions of subsets do not differ too much. Interestingly, the precise bounds depend on the Assouad and lower dimensions of the set, which give information about its extremal scaling properties. From this result we deduce a condition for the $\Phi$- and $\Phi_1$-intermediate dimensions to coincide for all subsets with finite Assouad dimension (Proposition~\ref{hardcomparisoncor}~(ii)). Specialising to the $\theta$-intermediate dimensions gives a continuity result (Theorem~\ref{intermediatects}) and sharp general lower bound (Proposition~\ref{ffkgeneralbounds}) which are proved directly in~\cite{Banaji2021moran} and improve bounds in \cite{Falconer2020,Falconer2021-2}. %
 Notably, the lower bound is strictly positive for all $\theta\in (0,1]$ if the box dimension of the set is positive; there is a `mutual dependency' (Proposition~\ref{mutual}) between the box and intermediate dimensions (as in~\cite[(14.2.7)]{Falconer2021-2}). %

In Section~\ref{holdersection} we prove H{\"o}lder distortion estimates for the $\Phi$-intermediate dimensions (Theorem~\ref{holder}) which, interestingly, are different from the standard $\dim f(F) \leq \alpha^{-1}\dim F$ bound for $\alpha$-H{\"o}lder images which holds for the Hausdorff, box and $\theta$-intermediate dimensions. The estimates imply bi-Lipschitz stability (Corollary~\ref{philipschitz}), which is an important property that most notions of dimension satisfy. This means that the $\Phi$-intermediate dimensions provide yet another invariant for the classification of subsets up to bi-Lipschitz image. 

In Section~\ref{masssection} we prove a mass distribution principle (Lemma~\ref{massdistprinc}) and a converse, a Frostman type lemma (Lemma~\ref{frostman}) for the $\Phi$-intermediate dimensions. The latter is an example of where the extension from Euclidean space to the more general metric spaces in which we work is non-trivial; we use an analogue of the dyadic cubes in general doubling metric spaces given in~\cite{Hytonen2010}. The mass distribution principle and Frostman type lemma combine to give Theorem~\ref{massfrostman}, a useful alternative definition of the $\Phi$-intermediate dimensions in terms of measures. We use this characterisation to prove Theorem~\ref{producttheorem} on the dimensions of product sets, giving new bounds in terms of the dimensions of the marginals, one of which we improve further in the case of self-products. In particular, ($\underline{\dim}^\Phi,\overline{\dim}_\mathrm{B}$) and ($\underline{\dim}_\theta,\overline{\dim}_\mathrm{B}$) satisfy the inequalities~\eqref{dimpairineq} that many `dimension pairs' satisfy, although our upper bound for $\overline{\dim}^\Phi (E \times F)$ is different to what might be expected. 
We also use the mass distribution principle to prove in Proposition~\ref{finitestability} that the lower versions of the intermediate and $\Phi$-intermediate dimensions are not finitely stable (in contrast to the upper versions).

Proposition~\ref{finitestability} also gives an example of a set to which the important result Theorem~\ref{recoverinterpolation} can be applied. Theorem~\ref{recoverinterpolation} shows that for every compact subset of an appropriate space there is a family of functions $\Phi$ which fully interpolate between the Hausdorff and box dimensions. Thus the $\Phi$-intermediate dimensions give finer geometric information about sets whose intermediate dimensions are discontinuous at $\theta=0$ by `recovering the interpolation' between Hausdorff and box dimension. Moreover, there exists a single family of $\Phi$ which interpolate for both the upper and lower versions of the dimensions, and whose dimensions vary monotonically for all sets, but in Proposition~\ref{interpolatenotcts} we show that it might not be possible to ensure that the dimensions vary continuously for all other sets.

\section{Preliminaries and definitions of dimensions}\label{prelimsection}%

 For $x \in X$ and $\delta>0$, we denote the open balls in $X$ and $F$ respectively by 
 \begin{gather*} 
 B(x,\delta) = B^X(x,\delta) \coloneqq \{ \, y \in X : d(x,y) < \delta \, \} , \\*
 B^F(x,\delta) \coloneqq \{ \, y \in F : d(x,y) < \delta \, \}, 
 \end{gather*}
 noting that these sets might have diameter less than $2\delta$. 
 We denote by $N_\delta(F)$ the smallest integer such that there exist $x_1,\dotsc,x_{N_\delta(F)} \in F$ such that 
 \[ F \subseteq \bigcup_{i = 1}^{N_\delta(F)} B(x_i,\delta/2).\] 
 The subset $F$ is \emph{totally bounded} if $N_\delta (F) < \infty$ for all $\delta > 0$. 
  In the definitions in this section, we will use the convention that $\inf \varnothing = \inf \{ \infty \} = \infty$. 
 Recall the following definition. 
 \begin{defn}\label{originalboxdef}
 The \emph{upper} and \emph{lower box dimension} of a non-empty, totally bounded subset $F$ of a metric space are defined respectively by \[\overline{\dim}_\mathrm{B} F \coloneqq \limsup_{\delta \to 0^+} \frac{\log (N_\delta (F))}{-\log \delta}; \qquad \underline{\dim}_\mathrm{B} F \coloneqq \liminf_{\delta \to 0^+} \frac{\log(N_\delta(F))}{-\log \delta}.\]
 \end{defn}
If $F \subset \mathbb{R}^n$ then there is an alternative definition of upper box dimension, 
 \begin{equation}\label{boxdef}
 \begin{aligned} 
 \overline{\dim}_{\mathrm B} F = \inf \{ \, s \geq 0 : &\mbox{ for all } \epsilon >0 \mbox{ there exists } \delta_0 \in (0,1] \mbox{ such that for all } \delta \in (0,\delta_0) \\
 &\mbox{ there exists a cover } \{U_1,U_2,\dotsc\} \mbox{ of } F \mbox{ such that } |U_i| = \delta \\
  &\mbox{ for all } i, \mbox{ and } \sum_i |U_i|^s \leq \epsilon \, \},
 \end{aligned}
 \end{equation}
 see~\cite[Chapter 2]{Falconer2014}. 
 One can define the Hausdorff dimension without using Hausdorff measure by 
\begin{equation}\label{hausdorffdef}
\begin{aligned} \dim_\mathrm{H} F = \inf \{ \, s \geq 0 : &\mbox{ for all } \epsilon >0 \mbox{ there exists a finite or countable cover } \\*
& \{U_1,U_2,\dotsc\} \mbox{ of } F \mbox{ such that } \sum_i |U_i|^s \leq \epsilon \,\},
\end{aligned}
\end{equation}%
see~\cite[Section~3.2]{Falconer2014}. 
Motivated by the similarity between~\eqref{boxdef} and~\eqref{hausdorffdef}, Falconer, Fraser and Kempton~\cite{Falconer2020} made the following definition, upon which our main Definition~\ref{maindefinition} for the $\Phi$-intermediate dimensions is based: 
\begin{defn}\label{d:intdimdef}
For $0 \leq \theta \leq 1$, the \emph{upper $\theta$-intermediate dimension} of a bounded set $F \subset \mathbb{R}^n$ is 
\begin{align*} \overline{\mbox{dim}}_\theta F = \inf \{ \, s \geq 0 : &\mbox{ for all } \epsilon >0 \mbox{ there exists } \delta_0 \in (0,1] \mbox{ such that for all } \delta \in (0,\delta_0) \\
&\mbox{ there exists a cover } \{U_1,U_2,\dotsc\} \mbox{ of } F \mbox{ such} \mbox{ that } \delta^{1/\theta}\leq |U_i| \leq \delta \\
&\mbox{ for all } i, \mbox{ and } \sum_i |U_i|^s \leq \epsilon \, \}.
\end{align*}
\end{defn}
Similarly, the \emph{lower $\theta$-intermediate dimensions} $\underline{\dim}_\theta F$ are defined in~\cite{Falconer2020}. It is also shown that for $F \subset \mathbb{R}^n$, the maps $\theta \mapsto \underline{\dim}_\theta F$ and $\theta \mapsto \overline{\dim}_\theta F$ are monotonically increasing in $\theta \in [0,1]$ and continuous in $\theta \in (0,1]$, but may be discontinuous at $\theta=0$. Banaji and Rutar~\cite{Banaji2021moran} have proved that a local derivative constraint gives a necessary and sufficient condition for a given function to be realised as the intermediate dimensions of a bounded subset of $\mathbb{R}^n$. 
 In this paper we often require the metric spaces we work with to satisfy certain properties. 
\begin{defn}\label{d:unifperf}%
For $c \in (0,1)$ we say a metric space $X$ is \emph{$c$-uniformly perfect} if for all $x \in X$ and $R \in \mathbb{R}$ such that $0 < R < |X|$, 
\[ B(x,R) \setminus B(x,cR) \neq \varnothing. \]
The space $X$ is \emph{uniformly perfect} if there exists $c \in (0,1)$ such that $X$ is $c$-uniformly perfect. 
\end{defn}
Intuitively, a metric space is uniformly perfect if it does not have islands which are very separated from the rest of the space. 

\begin{defn}
A metric space is said to be \emph{doubling} if there exists a constant $M \in \mathbb{N}$ (called the \emph{doubling constant}) such that for every $x \in X$ and $r>0$, there exists $x_1,\dotsc,x_M \in X$ such that $B(x,2r) \subseteq \bigcup_{i=1}^M B(x_i,r)$. 
\end{defn} 

The Assouad and lower dimensions, studied in detail in~\cite{Fraser2020}, are dual notions which give information about the `thickest' and `thinnest' part of a set respectively: 
\begin{defn}\label{assouaddef}
Suppose a subset $F$ of a metric space has more than one point. %
Then the \emph{Assouad dimension} of $F$ is defined by 
\begin{align*}
      \dim_\mathrm{A} F = \inf\{ \, 
      a : &\mbox{ there exists }C>0\mbox{ such that } N_r(B(x,R)\cap F) \leq C(R/r)^a \\
      &\mbox{ for all } x \in F \mbox{ and } 0<r<R \, \}. 
    \end{align*}
    The \emph{lower dimension} of $F$ is defined by
\begin{align*}
      \dim_\mathrm{L} F = \sup\{ \, 
      \lambda : &\mbox{ there exists } C>0\mbox{ such that } N_r(B(x,R)\cap F) \geq C(R/r)^\lambda \\
      &\mbox{ for all } x \in F \mbox{ and } 0<r<R\leq |F| \, \}. 
    \end{align*}
    \end{defn}
    
    In~\cite[Section~13.1.1]{Fraser2020} it is shown that a metric space $X$ with more than one point is uniformly perfect if and only if $0 < \dim_\mathrm{L} X$. Such a space cannot have any isolated points, so must be infinite. 
    It is also shown that a space $X$ is doubling if and only if $\dim_\mathrm{A} X < \infty$. 
    In this case we will see in Proposition~\ref{basicbounds} that all dimensions of every subset $F$ will be finite, as we will need to assume for many of the results in this paper. 
    A metric space is said to be \emph{Ahlfors regular} if there exists $s>0$, $C \geq 1$ and a Borel regular measure~$\mu$ supported on~$X$ such that $C^{-1} R^s \leq \mu(B_R) \leq C R^s$ for all closed balls $B_R$ of radius $0 < R < \mbox{diam}(X)$. By \cite[Corollary~14.15]{Heinonen2001metric}, every Ahlfors regular space with more than one point is uniformly perfect and doubling. 
    A familiar example of such a space is $\mathbb{R}^n$ with the Euclidean metric. 
    An example of such a space which is not bi-Lipschitz equivalent to any subset of $\mathbb{R}^n$ is the Heisenberg group with its usual Carnot-Carath{\'e}odory metric, see~\cite{LeDonne2015,Pansu1989,Semmes1996}. 
    
For the purposes of this paper, we make the following definition. 

 \begin{defn}\label{admissible}
 A function $\Phi \colon (0,\Delta) \to \mathbb{R}$ is \emph{admissible} if $\Phi$ is monotonic, $0<\Phi(\delta) \leq \delta$ for all $\delta \in (0,\Delta)$, and $\Phi(\delta)/\delta \to 0$ as $\delta \to 0^+$. 
 \end{defn}
 In some settings, for example when working with infinitely generated self-conformal sets in~\cite{Banaji2021}, it is convenient to assume that $\Phi(\delta)/\delta \to 0$ \emph{monotonically} as $\delta \to 0^+$. This is satisfied by many reasonable functions such as $\delta^{1/\theta}$ and $e^{-\delta^{-0.5}}$. 
  
  To minimise repetition, we make the following standing assumptions from this point onwards: 
  \begin{itemize}
    \item The letter $\Phi$ will represent an arbitrary admissible function (except in Proposition~\ref{p:nonadmissible} where we explore the conditions on $\Phi$). 
  \item The underlying metric space is denoted by $X$ (or sometimes $Y$), and will be assumed to have more than one point and be uniformly perfect. %
  The letter $c$ will usually denote the constant from Definition~\ref{d:unifperf}. %
  \item Subsets of $X$ are denoted by $F$ (or sometimes $E$ or $G$), and are assumed to be non-empty and totally bounded. 
  \end{itemize}
Using these conventions, and based on Definition~\ref{d:intdimdef}, we now make the main definition of this paper: 
\begin{defn}\label{maindefinition}
We define the \emph{upper $\Phi$-intermediate dimension} of a subset $F$ by 
\begin{align*} \overline{\dim}^\Phi F = \inf \{ \, s \geq 0 : &\mbox{ for all } \epsilon >0 \mbox{ there exists } \delta_0 \in (0,1] \mbox{ such that for all } \delta \in (0,\delta_0) \\
&\mbox{ there exists a cover } \{U_1,U_2,\dotsc\} \mbox{ of } F \mbox{ such that } \\
&\Phi(\delta)\leq |U_i| \leq \delta \mbox{ for all } i, \mbox{ and } \sum_i |U_i|^s \leq \epsilon \, \}.
\end{align*}
Similarly, we define the \emph{lower $\Phi$-intermediate dimension} of $F$ by 
\begin{align*} \underline{\dim}^\Phi F = \inf \{ \, s \geq 0 : &\mbox{ for all } \epsilon >0 \mbox{ and } \delta_0 \in (0,1] \mbox{ there exists } \delta \in (0,\delta_0) \mbox{ and a cover } \\
&\{U_1,U_2,\dotsc\} \mbox{ of } F \mbox{ such that } \Phi(\delta)\leq |U_i| \leq \delta \mbox{ for all } i,\\
&\mbox{ and } \sum_i |U_i|^s \leq \epsilon \,\}.
\end{align*}
If these two quantities coincide, we call the common value the $\Phi$\emph{-intermediate dimension} of $F$ and denote it by $\dim^\Phi F$. 
\end{defn}

In the above definition, the cover $\{U_i\}$ of $F$ is \emph{a priori} countable, but since it satisfies $0<\Phi(\delta)\leq |U_i|$ for all $i$, and $\sum_i |U_i|^s \leq \epsilon$, it must be finite. 
If $F$ were not totally bounded then the $\Phi$-intermediate dimensions of $F$ would be infinite according to Definition~\ref{maindefinition}. 
If $\theta \in (0,1)$ and $\Phi(\delta) = \delta^{1/\theta}$ for all $\delta \in [0,1]$, then $\overline{\dim}^\Phi F = \overline{\dim}_\theta F$ and $\underline{\dim}^\Phi F = \underline{\dim}_\theta F$ are the definitions of the \emph{upper} and \emph{lower intermediate dimensions} of $F$ at $\theta$, respectively. 
Set $\underline{\dim}_0 F = \overline{\dim}_0 F \coloneqq \dim_\mathrm{H} F$, $\underline{\dim}_1 F \coloneqq \underline{\dim}_\mathrm{B} F$, and $\overline{\dim}_1 F \coloneqq \overline{\dim}_\mathrm{B} F$. If $\underline{\dim}_\theta F = \overline{\dim}_\theta F$ then the common value is called the \emph{intermediate dimension} of $F$ at $\theta$ and is denoted by $\dim_\theta F$.

\section{Continuity and general bounds}\label{ctysection}

\subsection{The \texorpdfstring{$\Phi$}{Φ}-intermediate dimensions}\label{phiintermediatecty}

In this section we examine general bounds and continuity-like properties for the $\Phi$-intermediate dimensions. 
They satisfy the following inequalities, as with the intermediate dimensions. 

\begin{prop}\label{basicbounds}
For a subset $F$, 
\[ 0 \leq \dim_\mathrm{H} F \leq \underline{\dim}^\Phi F \leq \overline{\dim}^\Phi F \leq \overline{\dim}_\mathrm{B} F \leq \dim_\mathrm{A} F \leq \dim_\mathrm{A} X, \mbox{ and} \]
\[  \underline{\dim}^\Phi F \leq \underline{\dim}_\mathrm{B} F \leq \overline{\dim}_\mathrm{B} F. \]
\end{prop}

\begin{proof}
We first prove $\overline{\dim}^\Phi F \leq \overline{\dim}_\mathrm{B} F$. Recall that we denote by $c \in (0,1)$ a constant such that $X$ is $c$-uniformly perfect. Since $\Phi(\delta)/\delta \to 0$, there exists $\Delta \in (0,\min\{|X|,1\})$ such that $\Phi(\delta)/\delta < c/2$ for all $\delta \in (0,\Delta)$. 
Let $s > \overline{\dim}_\mathrm{B} F$ and $\epsilon > 0$. Let $t \in (\overline{\dim}_\mathrm{B} F, s)$, so we can reduce $\Delta$ further to assume that $\Delta < \epsilon^{\frac{1}{s-t}}$ %
and that for all $\delta \in (0,\Delta)$ there exists a cover of $F$ by $\delta^{-t}$ or fewer sets $\{U_i\}$, each having diameter at most $\delta$. We may assume without loss of generality that each $U_i$ intersects $F$. If $|U_i| \geq \delta/2$ then leave $U_i$ in the cover unchanged. If $|U_i| < \delta/2$, then fix $x_i \in U_i$ and $y_i \in B(x_i,\delta/2)\setminus B(x_i,c\delta/2)$; add the point $y_i$ to $U_i$, and call the resulting cover $\{V_i\}$. For each $i$, 
\[\Phi(\delta) \leq c\delta/2 \leq |V_i| \leq  \delta \]
by the triangle inequality. Moreover, 
\[\sum_i |V_i|^s \leq \delta^{-t}\delta^s < \delta_0^{s-t} < \epsilon. \]
Thus $\overline{\dim}^\Phi F \leq s$ by Definition~\ref{maindefinition}, so $\overline{\dim}^\Phi F \leq \overline{\dim}_\mathrm{B} F$, as required. 

The proof that $\underline{\dim}^\Phi F \leq \underline{\dim}_\mathrm{B} F$ is similar. %
Indeed, let $s' > \underline{\dim}_\mathrm{B} F$ and $\epsilon' > 0$. Let $t' \in (\underline{\dim}_\mathrm{B} F, s')$, so for all $\Delta' \in (0,\min\{(\epsilon')^{\frac{1}{s'-t'}},|X|,1\})$ there exists $\delta' \in (0,\Delta')$ and a cover of $F$ by $(\delta')^{-t'}$ or fewer sets, each having diameter at most $\delta'$. 
As above, we can use this cover to form a cover $\{V'_j\}$ which satisfies $\Phi(\delta') \leq |V'_j| \leq \delta'$ for all $j$ and $\sum_j |V'_j|^s < \epsilon'$. Therefore $\underline{\dim}^\Phi F \leq s'$, so $\underline{\dim}^\Phi F \leq \underline{\dim}_\mathrm{B} F$. 

 The inequalities $\dim_\mathrm{H} F \leq \underline{\dim}^\Phi F$, $\underline{\dim}^\Phi F \leq \overline{\dim}^\Phi F$ and $\underline{\dim}_\mathrm{B} F \leq \overline{\dim}_\mathrm{B} F$ follow directly from the definitions. The inequality $\overline{\dim}_\mathrm{B} F \leq \dim_\mathrm{A} F$ holds by fixing $R = |F|$ in Definition~\ref{assouaddef}. The inequality $\dim_\mathrm{A} F \leq \dim_\mathrm{A} X$ follows from Definition~\ref{assouaddef} since $F \subseteq X$. 
\end{proof}

We assume that the ambient metric space $X$ is uniformly perfect with more than one point, and that $\Phi(\delta)/\delta \to 0$ as $\delta \to 0^+$, to ensure that Proposition~\ref{basicbounds} will hold and to avoid cases like the two-point metric space, which would have infinite intermediate and $\Phi$-intermediate dimensions according to Definition~\ref{maindefinition}. There is no general relationship between the lower box dimension and the upper intermediate dimensions. 
It follows from Proposition~\ref{basicbounds} that if $F \subset \mathbb{R}^n$ is non-empty and bounded then $\underline{\dim}^\Phi F \leq \overline{\dim}^\Phi F \leq n$, and if in addition $F$ is open with respect to the Euclidean metric then $\underline{\dim}^\Phi F = \overline{\dim}^\Phi F = n$, as one would expect. 
The dimensions satisfy the following basic properties. 
\begin{prop}\label{unprovedprop}
  \hfill \begin{enumerate}[label=(\roman*)]
\item Both $\overline{\dim}^\Phi$ and $\underline{\dim}^\Phi$ are \emph{increasing for sets}: if $E \subseteq F$ then $\overline{\dim}^\Phi E \leq \overline{\dim}^\Phi F$ and $\underline{\dim}^\Phi E \leq \underline{\dim}^\Phi F$. 
\item Both $\overline{\dim}^\Phi$ and $\underline{\dim}^\Phi$ are \emph{stable under closure}: $\overline{\dim}^\Phi F = \overline{\dim}^\Phi \overline{F}$ and $\underline{\dim}^\Phi F = \underline{\dim}^\Phi \overline{F}$. 
\end{enumerate}
\end{prop}
\begin{proof}
This is straightforward from the definition. 
\end{proof}

\begin{example}\label{dirichlet}
The set $F \coloneqq \mathbb{Q} \cap [0,1] \subset \mathbb{R}$ is countable, so $\dim_\mathrm{H} F = 0$, but $\underline{\dim}^\Phi F = \overline{\dim}^\Phi F = 1$ for every admissible $\Phi$, directly from Definition~\ref{maindefinition}. This demonstrates that: 
\begin{itemize}
\item The dimensions $\underline{\dim}^\Phi$ and $\overline{\dim}^\Phi$ are different from $\dim_\mathrm{H}$. 
\item There are subsets of $\mathbb{R}$, such as $F$, for which there does not exist a family of admissible functions for which the $\Phi$-intermediate dimensions interpolate between the Hausdorff and box dimensions of the set. This means that the assumption of compactness in Theorem~\ref{recoverinterpolation} cannot be removed in general. 
\item The dimensions $\underline{\dim}^\Phi$ and $\overline{\dim}^\Phi$ can take positive values for countable sets. 
\end{itemize}
\end{example}

We will need the following sufficient condition for the $\Phi$-intermediate dimension always to equal the box dimension. 
As an example, the function $\Phi(\delta) \coloneqq \frac{\delta}{-\log \delta}$ satisfies the assumptions of Proposition~\ref{whenequalsbox}. 
\begin{prop}\label{whenequalsbox}
Let $\Phi$ be an admissible function such that $\frac{\log \delta}{\log \Phi(\delta)} \to 1$ as $\delta \to 0^+$. Then for all subsets $F$, $\overline{\dim}^\Phi F = \overline{\dim}_\mathrm{B} F$ and $\underline{\dim}^\Phi F = \underline{\dim}_\mathrm{B} F$. 
\end{prop}

\begin{proof}
We prove that $\overline{\dim}^\Phi F = \overline{\dim}_\mathrm{B} F$; the proof of $\underline{\dim}^\Phi F = \underline{\dim}_\mathrm{B} F$ is similar. %
Assume (for the purpose of obtaining a contradiction) that $\overline{\dim}^\Phi F < \overline{\dim}_\mathrm{B} F$, and let $s,t \in \mathbb{R}$ be such that $\overline{\dim}^\Phi F < s < t < \overline{\dim}_\mathrm{B} F$. Then for all sufficiently small $\delta$ there exists a cover $\{U_i\}$ of $F$ such that $\Phi(\delta) \leq |U_i| \leq \delta$ for all $i$, and $\sum_i |U_i|^s \leq 1$. Therefore 
\begin{equation*}
N_\delta(F) \delta^t \leq \sum_i \delta^t \frac{|U_i|^s|U_i|^{t-s}}{|U_i|^t} \leq \sum_i \delta^t \frac{|U_i|^s \delta^{t-s}}{(\Phi(\delta))^t}  \leq \left(\frac{\delta^{1+(t-s)/t}}{\Phi(\delta)}\right)^t, 
 \end{equation*}
 which converges to 0 as $\delta \to 0^+$. %
 This contradicts $t < \overline{\dim}_\mathrm{B} F$ and completes the proof. 
\end{proof}

We now consider continuity-like results for the $\Phi$-intermediate dimensions. 
The main such result is Theorem~\ref{maincty}, which roughly implies that if two admissible functions $\Phi$ and $\Phi_1$ are in a quantitative sense `close' to each other, then the $\Phi$ and $\Phi_1$-intermediate dimensions of sets whose Assouad dimension is not too large do not differ greatly. In a similar spirit, quantitative continuity results have been proven for the intermediate dimensions in $\mathbb{R}^n$, for example~\cite[Proposition~2.1]{Falconer2020},~\cite[(14.2.2)]{Falconer2021-2} and \cite[Theorem~2.6]{Banaji2021moran}. 
To obtain bounds involving the lower dimension, we will use the following definition: the \emph{lower dimension} of a Borel probability measure $\mu$ is  
\begin{align*}
 \dim_\mathrm{L} \mu \coloneqq \sup \left\{ \, \lambda \geq 0 : \right. &\left.\mbox{there exists } A > 0 \mbox{ such that if } 0 < r < R \leq |\mbox{supp}(\mu)| \right. \\*
&\left. \mbox{  and } x \in \mbox{supp}(\mu) \mbox{ then } \frac{\mu(B(x,R))}{\mu(B(x,r))} \geq A \left(\frac{R}{r}\right)^\lambda \right\}. 
\end{align*}
A measure $\mu$ is said to be \emph{doubling} if there exists $M \geq 1$, called the \emph{doubling constant}, such that $\mu(B(x,2r)) \leq M \mu(B(x,r))$ for all $x \in \mbox{supp}(\mu)$ and $r > 0$. 
For further details we refer the reader to~\cite[Section~4.1]{Fraser2020}. 
\begin{thm}\label{maincty}
Let $\Phi$ and $\Phi_1$ be admissible functions. Let $F$ be a subset satisfying $0 < \dim_\mathrm{A} F < \infty$,  and assume that $\overline{F}$ is complete. %
Suppose that $0 < \overline{\dim}^\Phi F < \dim_\mathrm{A} F$, and let $\eta \in [0,\dim_\mathrm{A} F - \overline{\dim}^\Phi F)$. %
Define 
\begin{equation}\label{e:definelambdaalpha}
\gamma \coloneqq \frac{\overline{\dim}^\Phi F - \dim_\mathrm{L} F}{\overline{\dim}^\Phi F + \eta - \dim_\mathrm{L} F}; \qquad \alpha \coloneqq \frac{\dim_\mathrm{A} F-\overline{\dim}^\Phi F}{\dim_\mathrm{A} F-\overline{\dim}^\Phi F - \eta}. 
\end{equation}
If 
\begin{equation}\label{continuityassumption}
 \Phi_1\left(\delta\right) \leq (\Phi(\delta^{1/\alpha}))^{\gamma}
 \end{equation}
 for all sufficiently small $\delta > 0$, then $\overline{\dim}^{\Phi_1} F \leq \overline{\dim}^\Phi F + \eta$. 
The same holds with $\overline{\dim}$ replaced by $\underline{\dim}$ throughout. 
\end{thm}

By a similar argument, if we only assume that $\Phi_1\left(\delta\right) \leq (\Phi(\delta^{1/\alpha}))^{\gamma}$ (with $\gamma$ and $\alpha$ as in~\eqref{e:definelambdaalpha}) holds only for a sequence of $\delta \to 0^+$, then we can only conclude $\underline{\dim}^{\Phi_1} F \leq \overline{\dim}^\Phi F + \eta$. 

\begin{proof}
Without loss of generality assume $\eta > 0$, so $\gamma < 1 < \alpha$. The idea of the proof is to convert a cover for the interval $[\Phi(\delta),\delta]$ into a cover for $[\Phi_1(\delta^\alpha),\delta^\alpha]$. %
We do this by using the Assouad dimension to replace sets which are too large with sets of size $\delta^\alpha$ (corresponding to indices $I_1$). We use the lower dimension to replace sets which are too small with sets of size $(\Phi(\delta))^{\gamma}$ (corresponding to indices $I_3$). 
We have chosen the parameters $\gamma$ and $\alpha$ so that the `cost' of each of these actions in terms of how much the dimension can increase is the same, namely $\eta$. This is similar to the strategy for the proof of the bound~\cite[Theorem~2.6]{Banaji2021moran} for the intermediate dimensions in $\mathbb{R}^n$. 

Without loss of generality we assume that $F$ is closed. 
Now for $s \in (\overline{\dim}^\Phi F,\dim_\mathrm{A} F - \eta)$ let $s' \in (\overline{\dim}^\Phi F,s)$, $a > \dim_\mathrm{A} F$ and $\lambda < \dim_\mathrm{L} F$ satisfy 
\begin{equation}\label{ctyparameters}
\gamma (s + \eta - \lambda) - (s'-\lambda) > 0  \qquad \mbox{and} \qquad  a - s' - \alpha (a-s-\eta) > 0.
\end{equation}
Let $c \in (0,1/2)$ be such that $X$ is $c$-uniformly perfect. 
Fix $C\in (0,\infty)$ such that $N_r(B(x,R)\cap F) \leq C(R/r)^a$ for all $x \in F$ and $0<r<R$. 
Since $F$ is assumed to be complete, by~\cite[Theorem~3.2]{Kaenmaki2017} (which is very similar to the main result of~\cite{Bylund2000}), there exists a doubling Borel probability measure $\mu$ with $\mbox{supp}(\mu) = F$ and $\dim_{\mathrm L} \mu \in (\lambda,\dim_{\mathrm L} F]$. %
In particular, there exists $A \in (0,1)$ such that if $0 < r < R \leq |F|$ and $x \in X$ then 
\[ \frac{\mu(B(x,R))}{\mu(B(x,r))} \geq A \left( \frac{R}{r}\right)^{\lambda}.\]
Fix $M>1$ such that $\mu$ is $M$-doubling. 

Let $\epsilon > 0$. Choose $\Delta >0$ such that for all $\delta \in (0,\Delta)$ there exists a cover $\{U_i\}_{i \in I}$ of $F$ such that $\Phi(\delta) \leq |U_i| \leq \delta$ for all $i$, and 
\[ \sum_i |U_i|^{s'} \leq  (c^{-(s+\eta)} M^2 A^{-1} 10^{s+\eta} + 3^{s+\eta} + 2^a C)^{-1} \epsilon. \]
We may reduce $\Delta$ to assume that~\eqref{continuityassumption} and $\delta/\Phi_1(\delta) \geq 5/c$ hold for all $\delta \in (0,\Delta)$, and $\Delta < 1$, $\Delta<|X|$. 
Write $I$ as a disjoint union $I = I_1 \cup I_2 \cup I_3$ where 
\begin{align*} 
I_1 &\coloneqq \{ \, i \in I : \Phi(\delta) \leq |U_i| < \Phi_1(\delta^\alpha) \, \} \\
I_2 &\coloneqq \{ \, i \in I : \Phi_1(\delta^\alpha) \leq |U_i| \leq \delta^\alpha/2 \, \} \\
I_3 &\coloneqq \{ \, i \in I : \delta^\alpha/2 < |U_i| \leq \delta \, \},
\end{align*}
noting that some of these sets may be empty. %
Let $z_1,\dotsc,z_K$ be a maximal $4\Phi_1(\delta^\alpha)$-separated subset of 
\[ F \setminus \left( \bigcup_{i \in I_2 \cup I_3} \mathcal{S}_{\Phi_1(\delta^\alpha)}(U_i)\right),\]
where $\mathcal{S}_r(U) \coloneqq \cup_{x \in U} B(x,r)$ is the $r$-neighbourhood of $U$. 

For each $k \in I_3$ pick ${x_{k,1},\dotsc,x_{k,\lfloor C(2|U_k|/\delta^\alpha)^a \rfloor} \in F}$ such that 
\[ \mathcal{S}_{\Phi_1(\delta^\alpha)}(U_k) \cap F \subseteq \bigcup_{l=1}^{\lfloor 2^a C|U_k|^a\delta^{-a\alpha} \rfloor} B(x_{k,l},\delta^\alpha/2).\]
Define 
\begin{align*}
\mathcal{U}_1 &\coloneqq \{ \, B(z_m,5\Phi_1(\delta^\alpha)/c) : 1 \leq m \leq K \, \}, \\
\mathcal{U}_2 &\coloneqq \{ \, \mathcal{S}_{\Phi_1(\delta^\alpha)} (U_j) : j \in I_2 \, \}, \\
\mathcal{U}_3 &\coloneqq \bigcup_{k \in I_3} \{ \,  B(x_{k,l},\delta^\alpha/2) : 1 \leq l \leq \lfloor 2^a C|U_l|^a\delta^{-a\alpha} \rfloor \, \}. 
\end{align*}
Then $\mathcal{U}_1 \cup \mathcal{U}_2 \cup \mathcal{U}_3$ is a cover of $F$, and for sufficiently small $\delta$ the diameter of each covering set lies in the interval $[\Phi_1(\delta^\alpha),\delta^\alpha]$. 

We bound the $(s+\eta)$-powers of the diameters of each part of the cover separately. First consider $\mathcal{U}_1$. 
For $m \in \{1,\dotsc,K\}$ let $J_m \coloneqq \{\, i \in I_1 : U_i \cap B(z_m,\Phi_1(\delta^\alpha)) \neq \varnothing \, \}$. 
If $i \in J_m$, let $u_{i,m} \in U_i \cap B(z_m,\Phi_1(\delta^\alpha))$. Then 
\begin{align*}
 \mu(U_i) \leq \mu(B(u_{i,m},2|U_i|)) &\leq A^{-1} \mu(B(u_{i,m},2\Phi_1(\delta^\alpha))) \left(\frac{\Phi_1(\delta^\alpha)}{|U_i|}\right)^{-\lambda} \\*
&\leq M^2 A^{-1} \mu(B(z_m,\Phi_1(\delta^\alpha))) \left(\frac{\Phi_1(\delta^\alpha)}{|U_i|}\right)^{-\lambda}. 
\end{align*} 
Therefore 
\[ \mu(B(z_m,\Phi_1(\delta^\alpha))) \leq \sum_{i \in J_m} \mu(U_i) \leq M^2 A^{-1} \mu(B(z_m,\Phi_1(\delta^\alpha))) \cdot (\Phi_1(\delta^\alpha))^{-\lambda} \cdot \sum_{i \in J_m} |U_i|^{\lambda}. \]
Since $\mbox{supp}(\mu) = F$, we can cancel through by the positive number $\mu(B(z_m,\Phi_1(\delta^\alpha)))$. 
Note also that if $i \in I_1$ then there is at most one $m$ for which $U_i \cap B(z_m,\Phi_1(\delta^\alpha)) \neq \varnothing$. Therefore 
\begin{align*}
\sum_{U \in \mathcal{U}_1} |U|^{s+\eta} &\leq K (10c^{-1} \Phi_1(\delta^\alpha))^{s+\eta} \\
&\leq c^{-(s+\eta)} M^2 A^{-1} 10^{s + \eta} (\Phi_1(\delta^\alpha))^{s + \eta - \lambda} \sum_{i \in I} |U_i|^{\lambda} \\
&\leq c^{-(s+\eta)} M^2 A^{-1} 10^{s + \eta} (\Phi_1(\delta^\alpha))^{s + \eta - \lambda} (\Phi(\delta))^{-(s' - \lambda)} \sum_{i \in I} |U_i|^{s'} \\
&\leq c^{-(s+\eta)} M^2 A^{-1} 10^{s + \eta} (\Phi(\delta))^{\gamma(s + \eta - \lambda) - (s' - \lambda)} \sum_{i \in I} |U_i|^{s'} \\
&< c^{-(s+\eta)} M^2 A^{-1} 10^{s + \eta} \sum_{i \in I} |U_i|^{s'}, 
\end{align*}
where we used~\eqref{ctyparameters} in the last step. 

For $\mathcal{U}_2$, 
\begin{equation*} 
\sum_{U \in \mathcal{U}_2} |U|^{s + \eta} \leq \sum_{j \in I_2} (3 |U_j|)^{s + \eta} \leq 3^{s+\eta} \sum_{j \in I} |U_j|^{s'}. 
\end{equation*}
Finally, consider $\mathcal{U}_3$. Since $|U_k| \leq \delta$ for $k \in I_3$, 
\begin{align*}
\sum_{k \in I_3} \sum_{l = 1}^{\lfloor 2^a C|U_k|^a\delta^{-a\alpha} \rfloor} |B(x_{k,l},\delta^\alpha/2)|^{s + \eta} &\leq \sum_{k \in I_3} 2^a C|U_k|^a \delta^{-a\alpha} \delta^{\alpha(s + \eta)} \\
&\leq 2^a C \delta^{-a\alpha + \alpha(s + \eta) + a-s'} \sum_{k \in I_3} |U_k|^{s'} \\
&\leq 2^a C \sum_{k \in I} |U_k|^{s'}. 
\end{align*}
Bringing the above bounds together, for all $\delta \in (0, \Delta)$, 
\[ \sum_{U \in \mathcal{U}_1 \cup \mathcal{U}_2 \cup \mathcal{U}_3} |U|^{s + \eta} \leq (c^{-(s+\eta)} M^2 A^{-1} 10^{s+\eta} + 3^{s+\eta} + 2^a C) \sum_{i \in I} |U_i|^{s'} \leq \epsilon. \]
It follows that $\overline{\dim}^{\Phi_1} F \leq s + \eta$, as required. 
The proof for when $\overline{\dim}$ is replaced by $\underline{\dim}$ is similar. 
\end{proof} 

The following is a similar result for the case when the $\Phi$-intermediate dimension of $F$ is $0$. 

\begin{prop}\label{zerocontinuityprop}
Let $\Phi,\Phi_1$ be admissible functions, assume $0 < \dim_\mathrm{A} F < \infty$, let $\eta \in (0,\dim_\mathrm{A} F)$, and let $b>0$. 
If for all sufficiently small $\delta$, 
\begin{equation}\label{zerocontinuityassumption} 
\Phi_1(\delta) \leq \left(\Phi(\delta^{1/\alpha})\right)^b 
\qquad \mbox{where} \qquad \alpha = \alpha(\eta) \coloneqq \frac{\dim_\mathrm{A} F}{\dim_\mathrm{A} F - \eta}
\end{equation}
holds, then if $\underline{\dim}^\Phi F = 0$ then $\underline{\dim}^{\Phi_1} F \leq \eta$, and if $\overline{\dim}^\Phi F = 0$ then $\overline{\dim}^{\Phi_1} F \leq \eta$. 
If we assume only that~\eqref{zerocontinuityassumption} holds for a subsequence of $\delta \to 0^+$, then if $\overline{\dim}^\Phi F = 0$ then $\underline{\dim}^{\Phi_1} F \leq \eta$. 
\end{prop}

\begin{proof} 
This is a straightforward modification of the proof of Theorem~\ref{maincty}. A cover for $[\Phi(\delta),\delta]$ is converted into a cover for $[\Phi_1(\delta^\alpha),\delta^\alpha]$ by breaking up the largest sets using the Assouad dimension of $F$, and fattening the smallest sets. 
The details are left to the reader. 
\end{proof} 
In particular, if $\Phi_1(\delta) \leq (\Phi(\delta))^b$ holds for some $b>0$ and all sufficiently small $\delta$, then $\overline{\dim}^{\Phi} F = 0$ implies $\overline{\dim}^{\Phi_1} F = 0$. 
The following Corollary of Theorem~\ref{maincty} and Proposition~\ref{zerocontinuityprop} says that if the underlying metric space is doubling, then if $\Phi$ and $\Phi_1$ are `close' in a way that depends only on $X$, then the difference between the $\Phi$- and $\Phi_1$-intermediate dimensions of subsets will be small, independently of the particular subset. %
\begin{cor}\label{indepctycor}
Let $X$ be a doubling metric space and suppose $F \subseteq X$ is bounded. If 
\begin{equation}\label{indepctyassumption}
\Phi_1\left(\delta^{\frac{\dim_\mathrm{A} X}{\dim_\mathrm{A} X - \eta}}\right) \leq (\Phi(\delta))^{\frac{\dim_\mathrm{A} X}{\dim_\mathrm{A} X + \eta}}
\end{equation}
holds for all sufficiently small $\delta$, then if $\overline{\dim}^\Phi F < \dim_\mathrm{A} F$ and $\eta \in [0,\dim_\mathrm{A} F - \overline{\dim}^\Phi F)$ then $\overline{\dim}^{\Phi_1} F \leq \overline{\dim}^\Phi F + \eta$, and the same holds with $\overline{\dim}$ replaced by $\underline{\dim}$ throughout. 
If we only assume that~\eqref{indepctyassumption} holds for a subsequence of $\delta \to 0^+$, and if $\overline{\dim}^\Phi F < \dim_\mathrm{A} F$ and $\eta \in [0,\dim_\mathrm{A} F - \overline{\dim}^\Phi F)$, then $\underline{\dim}^{\Phi_1} F \leq \overline{\dim}^\Phi F + \eta$. 
\end{cor} 

\begin{proof} 
Using notation from~\eqref{e:definelambdaalpha}, by Proposition~\ref{basicbounds}, 
\[ \gamma \leq \frac{\dim_\mathrm{A} X}{\dim_\mathrm{A} X + \eta}   \leq 1 \leq \frac{\dim_\mathrm{A} X}{\dim_\mathrm{A} X - \eta} \leq \alpha, \]
so the result follows from Theorem~\ref{maincty} in the cases $\overline{\dim}^\Phi F > 0$ and $\underline{\dim}^\Phi F > 0$, and from Proposition~\ref{zerocontinuityprop} in the cases $\overline{\dim}^\Phi F = 0$ and $\underline{\dim}^\Phi F = 0$. 
\end{proof} 

We write $\Phi_1 \preceq \Phi_2$ if $\overline{\dim}^{\Phi_1} F \leq \overline{\dim}^{\Phi_2} F$ and $\underline{\dim}^{\Phi_1} F \leq \underline{\dim}^{\Phi_2} F$ for all subsets $F$ with $\dim_{\mathrm{A}} F < \infty$ of every underlying space $X$. If $\Phi_1 \preceq \Phi_2$ and $\Phi_2 \preceq \Phi_1$, write $\Phi_1 \equiv \Phi_2$. 
Corollary~\ref{hardcomparisoncor} gives a condition for the dimensions to coincide for all sets. 

\begin{cor}\label{hardcomparisoncor}
Let $\Phi,\Phi_1$ be admissible functions. 
\begin{enumerate}[label=(\roman*)]
\item If for all $\alpha \in (1,\infty)$ there exists $\Delta > 0$ such that for all $\delta \in (0,\Delta)$ we have %
\begin{equation}\label{hardcomparisoneqn}
 \Phi_1(\delta) \leq (\Phi(\delta^{1/\alpha}))^{1/\alpha}
 \end{equation} 
 (noting that this will be the case if, for example, there exists $C \in (0,\infty)$ such that ${\limsup_{\delta \to 0^+} \frac{\Phi_1(C \delta)}{\Phi(\delta)} < \infty}$), then $\Phi_1 \preceq \Phi$. 
 If we only assume that for all $\alpha \in (1,\infty)$ and $\delta_0 > 0$ there exists $\delta \in (0,\delta_0)$ such that~\eqref{hardcomparisoneqn} holds, then we can only conclude that $\underline{\dim}^{\Phi_1} F \leq \overline{\dim}^\Phi F$ for every subset $F$ with finite Assouad dimension. 
 
 \item If for all $\alpha \in (1,\infty)$ there exists $\Delta > 0$ such that for all $\delta \in (0,\Delta)$,  
 \begin{equation}\label{hardequality}
  (\Phi(\delta^\alpha))^\alpha \leq \Phi_1(\delta) \leq (\Phi(\delta^{1/\alpha}))^{1/\alpha}
  \end{equation}
 holds, then $\Phi \equiv \Phi_1$. 
 \end{enumerate}
 \end{cor}
 
 \begin{proof} 
In the cases $\overline{\dim}^\Phi F = 0$ and $\overline{\dim}^\Phi F = \dim_\mathrm{A} F$, (i) follows from Propositions~\ref{zerocontinuityprop} and~\ref{basicbounds}. If $0 < \overline{\dim}^\Phi F < \dim_\mathrm{A} F$ then for all $\eta \in [0,\dim_\mathrm{A} F - \overline{\dim}^\Phi F)$, by the case of~\eqref{hardcomparisoneqn} with 
 \[ \alpha \coloneqq \min\left\{\frac{\dim_\mathrm{A} F-\overline{\dim}^\Phi F}{\dim_\mathrm{A} F-\overline{\dim}^\Phi F - \eta},\frac{\overline{\dim}^\Phi F + \eta}{\overline{\dim}^\Phi F}\right\}, \]%
 it follows that $\overline{\dim}^{\Phi_1} F \leq \overline{\dim}^\Phi F + \eta$ by Theorem~\ref{maincty}. Since $\eta$ was arbitrary, $\overline{\dim}^{\Phi_1} F \leq \overline{\dim}^\Phi F$. 
 Similarly, in all cases $\underline{\dim}^{\Phi_1} F \leq \underline{\dim}^\Phi F$, so $\Phi_1 \preceq \Phi$. 
 The case when we only assume~\eqref{hardequality} along a subsequence is proved similarly, and~(ii) follows from~(i). 
 \end{proof}

We now use Corollary~\ref{hardcomparisoncor} to explore the conditions that can be imposed on the function $\Phi$, and show that nothing is really lost by only considering functions which are strictly increasing, invertible and continuous. %

\begin{prop}
For every admissible function $\Phi$ there exists an admissible function $\Phi_1 \colon (0,1) \to (0,1)$ that is a strictly increasing, $C^\infty$ diffeomorphism, such that $\Phi \equiv \Phi_1$. 

\end{prop}

\begin{proof}

Fix $N \in \mathbb{N}$ such that $\Phi$ is positive and increasing on $(0,2^{-N}]$ with $\Phi(2^{-N}) < 1$. We construct a strictly increasing function $\Phi_2 \colon (0,1] \to (0,1]$ by defining $\Phi_2$ to be linear on $[2^{-N},1]$ with $\Phi_2(2^{-N}) = \Phi(2^{-N})$ and $\Phi_2(1)=1$ and defining $\Phi_2$ inductively on $(0,2^{-N})$ as follows. 
Suppose we have defined $\Phi_2$ on $[2^{-n},1]$ for some $n \geq N$. If $\Phi(2^{-n-1}) < \Phi_2(2^{-n})$ then define $\Phi_2(2^{-n-1}) = \Phi(2^{-n-1})$ and $\Phi_2$ linear on $[2^{-n-1},2^{-n}]$. If, on the other hand, $\Phi(2^{-n-1}) = \Phi_2(2^{-n})$, then let $m > n$ be the smallest integer such that $\Phi(2^{-m}) < \Phi(2^{-n})$, define $\Phi_2(2^{-m}) \coloneqq \max\{\Phi_2(2^{-n})/2,\Phi(2^{-m})\}$, and define $\Phi_2$ to be linear on $[2^{-m},2^{-n}]$. Then by construction $\Phi_2$ is strictly increasing on $(0,1]$ with $\Phi_2(\delta/4) \leq \Phi(\delta)$ and $2\Phi_2(2\delta) \geq \Phi(\delta)$ for all $\delta \in (0,2^{-N-1})$. 
Each of the countably many points of non-differentiability of $\Phi_2$ can be locally made smooth to give an admissible function $\Phi_1 \colon (0,1) \to (0,1)$ that is $C^\infty$ on $(0,1)$, still strictly increasing, and such that $\Phi_2(\delta)/2 \leq \Phi_1(\delta) \leq 2\Phi_2(\delta)$ for all $\delta \in (0,2^{-N})$. Then 
\[ \Phi_1(\delta)/\delta \leq 2\Phi_2(\delta)/\delta \leq 2\Phi(4\delta)/\delta = 8\Phi(4\delta)/(4\delta) \xrightarrow[\delta \to 0^+]{} 0, \]
so $\Phi_1$ is admissible. Moreover, 
\[ \Phi(\delta/2)/4 \leq \Phi_2(\delta)/2 \leq \Phi_1(\delta) \leq 2\Phi_2(\delta) \leq 2\Phi(4\delta) \]
for all $\delta \in (0,2^{-N-3})$, so $\Phi_1 \equiv \Phi$ by Corollary~\ref{hardcomparisoncor}~(ii). By the smooth inverse function theorem, $\Phi_1$ has a $C^\infty$ inverse, as required. 
\end{proof}
The following proposition shows that the assumption that $\Phi$ is monotonic and strictly positive does not really lose anything. 
\begin{prop}\label{p:nonadmissible}
Let $\Phi \colon (0,\Delta] \to [0,\infty)$ be \emph{any} function (not necessarily monotonic) such that $\Phi(\delta)/\delta \to 0$ as $\delta \to 0^+$, and define the $\Phi$-intermediate dimensions as in Definition~\ref{maindefinition}. Let $F$ be a subset. %
\begin{enumerate}[label=(\roman*)]
\item 
If $\Phi_1$ is defined by $\Phi_1(\delta) \coloneqq \sup\{\,\Phi(\delta') : \delta' \in [0,\delta] \, \}$ then $\overline{\dim}^\Phi F = \overline{\dim}^{\Phi_1} F$. 
\item \begin{enumerate}[label=(\arabic*)]
\item If there is a sequence of $\delta \to 0^+$ for which $\Phi(\delta) = 0$ then $\underline{\dim}^\Phi F = \dim_\mathrm{H} F$. 

\item Suppose $\Phi(\delta)>0$ for all $\delta \in (0,\Delta)$ but for all $\delta_2 \in (0,\Delta)$ there exists $\delta_3 \in (0,\delta_2)$ such that $\inf\{\,\Phi(\delta) : \delta \in [\delta_3,\delta_2] \, \} = 0$. Then if $F$ is compact then $\underline{\dim}^\Phi F = \dim_\mathrm{H} F$. In particular, if $F$ is any non-empty, bounded subset of $X=\mathbb{R}^n$ then $\underline{\dim}^\Phi F = \dim_\mathrm{H} \overline{F}$. %

\item If $\Phi_2 \colon (0,\Delta) \to \mathbb{R}$ defined by $\Phi_2(\delta) \coloneqq \inf\{\,\Phi(\delta') : \delta' \in [\delta,\Delta] \, \}$ is positive for all $\delta \in (0,\Delta)$, then $\underline{\dim}^\Phi F = \underline{\dim}^{\Phi_2} F$. 
\end{enumerate}
\end{enumerate}
\end{prop}

\begin{proof}
We may assume that $\Delta < \min\{1,|X|\}$ and that $\Phi(\delta) \leq (1+2/c)^{-1} \delta$ for all $\delta \in (0,\Delta)$. In the proofs of the different parts of the proposition, the same symbols may take different values. 

(i) For all $\delta \in (0,\Delta)$, 
\[ \Phi_1(\delta)/\delta = \sup\{ \, \Phi(\delta')/\delta : \delta' \in (0,\delta] \, \} \leq \sup\{ \, \Phi(\delta')/\delta' : \delta' \in (0,\delta] \, \} \xrightarrow[\delta \to 0^+]{} 0, \]
and $\Phi_1(\delta)$ is monotonic, so $\Phi_1$ is admissible. 
Also, $\Phi(\delta)\leq \Phi_1(\delta)$, so $\overline{\dim}^\Phi F \leq \overline{\dim}^{\Phi_1} F$. 
It remains to prove the reverse inequality. %
Let $s>\overline{\dim}^\Phi F$ and $\epsilon > 0$. Then there exists $\delta_0>0$ such that for all $\delta \in (0,\min\{\delta_0,\Delta\})$ there exists a cover $\{U_i\}$ of $F$ such that $\Phi(\delta) \leq |U_i| \leq \delta$ for all $i$, and 
\begin{equation}\label{firstphicompareeps}
\sum_i |U_i|^s \leq 2^{-s}(1+1/c)^{-s} \epsilon.
\end{equation}
 Then if $\delta' \in (0,\delta_0)$ then there exists $\delta \in (0,\delta']$ such that $\Phi(\delta)\geq \Phi_1(\delta')/2$. Let $\{U_i\}$ be the cover corresponding to $\delta$ as above. For each $i$, if $|U_i| \geq \Phi_1(\delta')$ then leave $|U_i|$ in the cover unchanged, noting that $\Phi_1(\delta') \leq |U_i| \leq \delta \leq \delta'$. If $\Phi_1(\delta') > |U_i|$, on the other hand, then fix $p_i \in U_i$, and $q_i \in X$ such that $\Phi_1(\delta') \leq d(p_i,q_i) \leq \Phi_1(\delta')/c$. Replace $U_i$ in the cover by $U_i \cup \{q_i\}$, and denote the new cover of $F$ by $\{V_i\}_i$. Then 
\[ \Phi_1(\delta') \leq d(p_i,q_i) \leq |U_i \cup \{q_i\}| < (1+1/c)\Phi_1(\delta') \leq \delta'. \] 
Also,
 \[ |U_i \cup \{q_i\}| \leq 2(1+1/c)\Phi(\delta) \leq 2(1+1/c)|U_i|.\]
Therefore 
\[ \sum_i |V_i|^s \leq \sum_i (2(1+1/c)|U_i|)^s = 2^s (1+1/c)^s \sum_i |U_i|^s \leq \epsilon \]
by~\eqref{firstphicompareeps}, so $\overline{\dim}^{\Phi_1} F \leq s$, hence $\underline{\dim}^{\Phi_1} F \leq \overline{\dim}^\Phi F$ as required. 

(ii) 
(1) Follows directly from~\eqref{hausdorffdef} and Definition~\ref{maindefinition}. 

(ii) (2) 
Assume that $F$ is compact. %
Let $s>\dim_\mathrm{H} F$, $\epsilon>0$ and $\delta_2 \in (0,1]$, so there exists $\delta_3 \in (0,\delta_2)$ such that $\inf\{\,\Phi(\delta) : \delta \in [\delta_3,\delta_2] \, \} = 0$. There exists a countable cover $\{U_i\}$ of $F$ such that $\sum_i |U_i|^{s} \leq \min\{\delta_3^{s},\epsilon\}$. In particular, $|U_i|\leq \delta_3$. Since $F$ is compact, there is a finite subcover $\{V_i\}$, so $\min_i\{|V_i|\}>0$, and each  $|V_i|\leq \delta_3$. Since $\inf\{\,\Phi(\delta) : \delta \in [\delta_3,\delta_2] \, \} = 0$, there exists $\delta_4 \in [\delta_3,\delta_2]$ such that $\Phi(\delta_4) \in (0,\min_i\{|V_i|\})$. Then $0 \leq \Phi(\delta_4) \leq \min_i\{|V_i|\}) \leq |V_i| \leq \delta_3 \leq \delta_4$ for each $i$, and $\sum_i |V_i|^{s} \leq \sum_i |U_i|^{s} \leq \epsilon$. As $\epsilon$ and $\delta_2$ were arbitrary, $\underline{\dim}^\Phi F \leq s$, so $\underline{\dim}^\Phi F = \dim_\mathrm{H} F$. 

(ii) (3) 
Clearly $\Phi_2$ is admissible and $\underline{\dim}^{\Phi_2} F \leq \underline{\dim}^{\Phi} F$, so it remains to prove the reverse inequality. Let $s > \underline{\dim}^{\Phi_2} F$ and $\epsilon > 0$. Let $\delta_1 > 0$ and let $\delta_0 \in (0,\Phi_2(\min\{\Delta,\delta_1\})/2)$. 
 Then there exists $\delta \in (0,\delta_0)$ and a cover $\{U_i\}$ of $F$ such that $\Phi_2(\delta) \leq |U_i| \leq \delta$ for all $i$, and 
 \begin{equation}\label{lastphicompareeps}
 \sum_i |U_i|^s \leq 2^{-s}(1+1/c)^{-s} \epsilon.
 \end{equation}
  By the definition of $\Phi_2$, there exists $\delta_2 \in [\delta,\Delta]$ such that $\Phi(\delta_2) < 2\Phi_2(\delta)$. But since $\Phi_2(\delta) \leq \Phi(\delta) \leq \delta_0 < \Phi_2(\min\{\Delta,\delta_1\})/2$, %
 it must be the case that $\delta_2 <\min\{\Delta,\delta_1\}$. %
 If $|U_i| \geq \Phi(\delta_2)$ then leave $U_i$ in the cover unchanged. If $|U_i| < \Phi(\delta_2)$ then fix $p_i \in U_i$ and $q_i \in X$ such that $\Phi(\delta_2) \leq d(p_i,q_i) \leq \Phi(\delta_2)/c$; replace $U_i$ in the cover with $U_i \cup \{q_i\}$ and call the new cover $\{V_i\}$. 
 Now, $\Phi(\delta_2) \leq |U_i \cup \{q_i\}| < \delta_2$. Also, $|U_i \cup \{q_i\}| \leq 2(1+1/c)\Phi_2(\delta) \leq 2(1+1/c)|U_i|$. 
 Therefore 
 \[ \sum_i |V_i|^s \leq \sum_i (2(1+1/c)|U_i|)^s = 2^s (1+1/c)^s \sum_i|U_i|^s \leq \epsilon, \]
 by~\eqref{lastphicompareeps}. It follows that $\underline{\dim}^{\Phi} F \leq s$, as required. 
\end{proof}

\subsection{The intermediate dimensions}\label{s:intcty}

In this section explore the consequences of general results proved for the $\Phi$-intermediate dimensions in Section~\ref{phiintermediatecty} for the special case of the $\theta$-intermediate dimensions. 
By Proposition~\ref{zerocontinuityprop}, the following mutual dependency between the box and intermediate dimensions holds. 

\begin{prop}\label{mutual}
If $0 < \overline{\dim}_\mathrm{B} F \leq \dim_\mathrm{A} F < \infty$ then $\overline{\dim}_\theta F > 0$ for all $\theta \in (0,1]$. The same holds with $\overline{\dim}$ replaced by $\underline{\dim}$ throughout. 
\end{prop}

\begin{proof}
Assume $0 < \dim_\mathrm{A} F < \infty$, let $\theta \in (0,1)$, and suppose that $\overline{\dim}_\theta F = 0$. Then if $\Phi(\delta) = \delta^{1/\theta}$ and $\Phi_1(\delta) = \delta/(-\log \delta)$ then $\Phi_1(\delta) \leq \delta = (\Phi(\delta))^{\theta}$, so $0 = \overline{\dim}^{\Phi_1} F = \overline{\dim}_\mathrm{B} F$ by Propositions~\ref{zerocontinuityprop} and~Proposition~\ref{whenequalsbox}. 
The proof for the lower versions of the dimensions is similar. 
\end{proof}

For subsets of Euclidean space, Proposition~\ref{mutual} also follows from~\cite[(14.2.7)]{Falconer2021-2}. 
Proposition~\ref{mutual} means that in order to check that the box dimension of a set is 0, it suffices to check the \emph{a priori} weaker condition that the $\theta$-intermediate dimension of the set is 0 at a small $\theta \in (0,1]$. 
It would be interesting to know if there are sets whose box dimension has resisted calculation by other methods but can be calculated in this way. 
Another mutual dependency result between different notions of dimension is that the upper box dimension of a set is 0 if and only if its Assouad spectrum and quasi-Assouad dimensions are 0, which follows from work in~\cite{Fraser2018-2,Garcia2021,Fraser2019-3}. 
 
 Theorem~\ref{intermediatects} is a quantitative continuity result for the intermediate dimensions which improves~\cite[Proposition~2.1]{Falconer2020} and~\cite[(14.2.2)]{Falconer2021-2}. The proof of~\cite[Proposition~2.1]{Falconer2020} involves breaking up the largest sets in the cover, while~\cite[(14.2.2)]{Falconer2021-2} is proved by `fattening' the smallest sets in the cover. The novelty in the proof of Theorem~\ref{maincty} (from which Theorem~\ref{intermediatects} follows) is to deal with the smallest and largest sets at the same time in such a way that the `cost' of each (in terms of how much the dimension can increase) is the same. %
 Banaji and Rutar recently gave a direct proof of Theorem~\ref{intermediatects} in the Euclidean setting in the proof of~\cite[Theorem~2.6]{Banaji2021moran}. 
 
 \begin{thm}\label{intermediatects}
If $0 \leq \dim_\mathrm{L} F < \dim_\mathrm{A} F < \infty$ %
and $0 < \theta \leq \phi \leq 1$ then 
 \[  \overline{\dim}_\theta F \leq \overline{\dim}_\phi F \leq \overline{\dim}_\theta F + \frac{(\overline{\dim}_\theta F - \dim_\mathrm{L} F) (\dim_\mathrm{A} F -\overline{\dim}_\theta F)}{\phi (\overline{\dim}_\theta F - \dim_\mathrm{L} F) + \theta( \dim_\mathrm{A} F -\overline{\dim}_\theta F)}(\phi-\theta). \] 
 The same holds with $\overline{\dim}$ replaced by $\underline{\dim}$ throughout. 
Furthermore, the functions $\theta' \mapsto \overline{\dim}_{\theta'} F$ and $\theta' \mapsto \underline{\dim}_{\theta'} F$ are continuous for $\theta' \in (0,1]$; indeed they are both Lipschitz on $[\theta,1]$ with Lipschitz constant $\frac{\dim_\mathrm{A} F}{4\theta}$. 
\end{thm}

Therefore by Lipschitz continuity and Rademacher's theorem, or alternatively by monotonicity and Lebesgue's theorem, the functions $\theta' \mapsto \overline{\dim}_{\theta'} F$ and $\theta' \mapsto \underline{\dim}_{\theta'} F$ are differentiable at Lebesgue-almost every $\theta' \in (0,1)$. %

\begin{proof}
We prove the version for $\overline{\dim}$; the version for $\underline{\dim}$ is similar. 
The inequality $\overline{\dim}_\theta F \leq \overline{\dim}_\phi F$ is immediate from the definitions. %
The only non-trivial case of the other inequality is when $0 < \theta < \phi \leq 1$ and $0 < \overline{\dim}_\theta F < \dim_\mathrm{A} F$. Define $\Phi(\delta) \coloneqq \delta^{1/\theta}$. If $\phi<1$, define $\Phi_1(\delta) \coloneqq \delta^{1/\phi}$, but if $\phi = 1$ then define $\Phi_1(\delta) \coloneqq \delta/(-\log \delta)$. Then $\overline{\dim}^\Phi F = \overline{\dim}_\theta F$ and $\overline{\dim}^{\Phi_1} F = \overline{\dim}_\phi F$. Define 
\[\eta \coloneqq \frac{(\overline{\dim}_\theta F - \dim_\mathrm{L} F) (\dim_\mathrm{A} F -\overline{\dim}_\theta F)}{\phi (\overline{\dim}_\theta F - \dim_\mathrm{L} F) + \theta( \dim_\mathrm{A} F -\overline{\dim}_\theta F)}(\phi-\theta). \] 
Using notation from~\eqref{e:definelambdaalpha}, a direct manipulation now shows that $\alpha/\phi = \gamma/\theta$. 
Therefore 
\[ \Phi_1(\delta^\alpha) \leq \delta^{\alpha/\phi} = \delta^{\gamma/\theta} = (\Phi(\delta))^{\gamma/\theta}.\] 
Thus $\overline{\dim}^{\Phi_1} F \leq \overline{\dim}^\Phi F + \eta$ by Theorem~\ref{maincty}, as required. 
To deduce Lipschitz continuity on $[\theta,1]$, note that if $0 < \theta \leq \theta' \leq \phi \leq 1$ then 
\begin{equation*}
\overline{\dim}_\phi F - \overline{\dim}_{\theta'} F 
\leq 
\frac{(\dim_{\mathrm{A}} F - \dim_{\mathrm{L}} F)/2)^2}{(\dim_{\mathrm{A}} F - \dim_{\mathrm{L}} F)\theta} (\phi - \theta')
\leq \frac{\dim_{\mathrm{A}} F - \dim_{\mathrm{L}} F}{4\theta}(\phi - \theta'). \qedhere
\end{equation*}
\end{proof}

Falconer noted that his continuity result~\cite[(14.2.2)]{Falconer2021-2} shows that $\frac{\overline{\dim}_\theta F}{\theta}$ and $\frac{\underline{\dim}_\theta F}{\theta}$ are monotonically decreasing in $\theta \in (0,1]$, so the graphs of $\theta \to \overline{\dim}_\theta F$ and $\theta \to \underline{\dim}_\theta F$ for $\theta \in (0,1]$ are starshaped with respect to the origin. 
 Corollary~\ref{starshaped} shows that in fact the graphs are \emph{strictly} starshaped, and every half-line from the origin in the first quadrant intersects the graphs in a single point. 

\begin{cor}\label{starshaped}
If $0 < \overline{\dim}_\mathrm{B} F \leq \dim_\mathrm{A} F < \infty$ then $(\overline{\dim}_\theta F)/\theta$ is strictly decreasing in $\theta \in (0,1]$. The same holds with $\overline{\dim}$ replaced by $\underline{\dim}$ throughout. 
\end{cor} 

\begin{proof} 

The only non-trivial case is when $\dim_{\mathrm L} F < \dim_{\mathrm A} F$. Suppose $0 < \theta < \phi \leq 1$. By Proposition~\ref{mutual}, $\overline{\dim}_\theta F > 0$, so by Theorem~\ref{intermediatects} and a direct algebraic manipulation, 
\[ \frac{\overline{\dim}_\phi F}{\phi} \leq \frac{1}{\phi}\left(     \overline{\dim}_\theta F + \frac{(\overline{\dim}_\theta F - \dim_\mathrm{L} F) (\dim_\mathrm{A} F -\overline{\dim}_\theta F)}{\phi (\overline{\dim}_\theta F - \dim_\mathrm{L} F) + \theta( \dim_\mathrm{A} F -\overline{\dim}_\theta F)}(\phi-\theta)   \right) < \frac{\overline{\dim}_\theta F}{\theta}. \qedhere \]
\end{proof}
If we know the value of $\overline{\dim}_\theta F$ or $\underline{\dim}_\theta F$ for one value of $\theta \in (0,1]$ then Theorem~\ref{intermediatects} gives an upper bound for $\overline{\dim}_\phi F$ or $\underline{\dim}_\phi F$ respectively, for all $\phi \in [\theta,1]$. 
Rearranging this bound gives 
\begin{align*}
 \overline{\dim}_\phi F &(\phi (\overline{\dim}_\theta F - \dim_\mathrm{L} F) + \theta( \dim_\mathrm{A} F -\overline{\dim}_\theta F)) \\
 &\leq \overline{\dim}_\theta F (\phi (\overline{\dim}_\theta F - \dim_\mathrm{L} F) + \theta( \dim_\mathrm{A} F -\overline{\dim}_\theta F)) \\*
 &+ (\overline{\dim}_\theta F - \dim_\mathrm{L} F) (\dim_\mathrm{A} F -\overline{\dim}_\theta F)(\phi-\theta). 
 \end{align*}
Expanding brackets, cancelling terms and rearranging, we obtain what can be thought of as a lower bound for
for $\overline{\dim}_\theta F$ in terms of $\overline{\dim}_\phi F$: 
\begin{equation}\label{ctylowerbound}
 \overline{\dim}_\theta F \geq \frac{ \theta \dim_{\mathrm A} F ( \overline{\dim}_\phi F - \dim_{\mathrm L} F ) +    \phi  \dim_{\mathrm L} F ( \dim_{\mathrm A} F -  \overline{\dim}_\phi F ) }{ \theta (\overline{\dim}_\phi F - \dim_{\mathrm L} F) + \phi ( \dim_{\mathrm A} F  - \overline{\dim}_\phi F)}. 
 \end{equation}
Of particular interest is the lower bound for the intermediate dimensions in terms of the box dimension, because the box dimension of many sets is known independently. 
The following bound is given for subsets of $\mathbb{R}^n$ in \cite[Corollary~2.8]{Banaji2021moran}. 

\begin{cor}\label{ffkgeneralbounds}
If $\dim_\mathrm{L} F < \dim_\mathrm{A} F < \infty$ then for all $\theta \in (0,1]$, 
\[ \overline{\dim}_\theta F \geq \frac{\theta \dim_{\mathrm A} F ( \overline{\dim}_{\mathrm B} F - \dim_{\mathrm L} F ) + \dim_{\mathrm L} F ( \dim_{\mathrm A} F -  \overline{\dim}_{\mathrm B} F ) }{ \theta (\overline{\dim}_{\mathrm B} F - \dim_{\mathrm L} F) +  ( \dim_{\mathrm A} F  - \overline{\dim}_{\mathrm B} F)}. \]
The same holds replacing $\overline{\dim}$ with $\underline{\dim}$ throughout. 
\end{cor}

\begin{proof}
Set $\phi = 1$ in~\eqref{ctylowerbound}. 
\end{proof}

We make several remarks about this bound. 
\begin{itemize}
\item In the bounds in Sections~\ref{phiintermediatecty} and~\ref{s:intcty}, if every instance of $\dim_{\mathrm L} F$ is replaced by $0$ and every instance of $\dim_{\mathrm A} F$ is replaced by $n$, then we obtain bounds which hold for all non-empty bounded $F \subset \mathbb{R}^n$. 

\item If $\dim_\theta F \in \{0,\dim_{\mathrm L} F, \dim_{\mathrm A} F\}$ for some $\theta \in (0,1]$ then $\dim_\theta F$ is constant on $(0,1]$. This extends results in~\cite{Falconer2020,Falconer2021-2,Banaji2021moran} to more general metric spaces. 

\item Assume $\dim_{\mathrm L} F < \dim_{\mathrm B} F < \dim_{\mathrm A} F$. Then one can differentiate the bound and show that it is real analytic, strictly increasing, strictly concave, and takes value $\dim_{\mathrm L} F$ at $\theta = 0$ and $\dim_{\mathrm B} F$ at $\theta = 1$. In this sense, it is a quantitative improvement of Proposition~\ref{mutual}. 

\item As $\dim_{\mathrm B} F$ approaches $\dim_{\mathrm A} F$ or $\dim_{\mathrm L} F$ respectively, so does the lower bound pointwise. 

\item The dimension theory of the self-affine Bedford--McMullen carpets has received considerable attention~\cite{Fraser2021-2}. Banaji and Kolossv\'ary recently proved a precise formula for the intermediate dimensions in~\cite{Banaji2021bedford}, which have previously been studied in~\cite[Section~4]{Falconer2020} and~\cite{Kolossvary2022bm}. For some carpets, in particular when the maps in the defining iterated function system are very unevenly distributed in the different columns, Corollary~\ref{ffkgeneralbounds} can give non-trivial information when $\theta$ is close to 1. 

\end{itemize}
We now show that the above bounds are sharp (in contrast to the bounds~\cite[Proposition~2.1]{Falconer2020} and~\cite[(14.2.7)]{Falconer2021-2}). Working in $\mathbb{R}$, for $p \in (0,\infty)$ let $F_p \coloneqq \{0\} \cup \{ \, n^{-p} : n \in \mathbb{N} \, \}$. It is straightforward to verify that $\dim_{\mathrm L} F_p = 0$, $\dim_\mathrm{A} F_p = 1$ and $\dim_\mathrm{B} F_p = \frac{1}{p+1}$. Falconer, Fraser and Kempton~\cite[Proposition~3.1]{Falconer2020} showed that $\dim_\theta F_p = \theta/(p+\theta)$ for all $\theta \in [0,1]$. 
Therefore if $0 < \theta \leq \phi \leq 1$ then by a direct algebraic manipulation, 
\begin{align*}
 \dim_\theta F_p &+ \frac{(\dim_\theta F - \dim_\mathrm{L} F) (\dim_\mathrm{A} F - \dim_\theta F)}{\phi (\dim_\theta F - \dim_\mathrm{L} F) + \theta( \dim_\mathrm{A} F - \dim_\theta F)}(\phi-\theta) \\
 &=  \frac{\theta}{p + \theta} + \frac{\frac{\theta}{p + \theta} \left(1-\frac{\theta}{p + \theta}\right)}{(\phi - \theta)\frac{\theta}{p + \theta} + \theta}(\phi - \theta) = \frac{\phi}{p+\phi} = \dim_\phi F_p,
\end{align*}
so the upper bound of Theorem~\ref{intermediatects} is attained. 
Similarly, this family of examples shows that the Lipschitz constant in~Theorem~\ref{intermediatects} and the lower bound Corollary~\ref{ffkgeneralbounds} cannot be improved in general. 
These bounds are also sharp for certain lattice sets which generalise the $F_p$ sets to higher dimensions. 
In~\cite[Proposition~3.8]{Banaji2021}, Banaji and Fraser used the bound in Corollary~\ref{ffkgeneralbounds} to calculate the intermediate dimensions of these lattice sets without needing to use a mass distribution argument as in the proof of~\cite[Proposition~3.1]{Falconer2020}. 
The bound can also be used to calculate the intermediate dimensions of the graph of the popcorn function~\cite{Banaji2022popcorn}. 

A certain converse to Theorem~\ref{intermediatects} was proved by Banaji and Rutar as the main result of~\cite{Banaji2021moran} using a Moran set construction. 
In particular, if $d \in \mathbb{N}$ and $0 \leq \lambda \leq \alpha \leq d$ and $h \colon [0,1] \to [0,d]$ is an \emph{arbitrary} increasing function satisfying 
\[ h(\phi) \leq h(\theta) + \frac{(h(\theta) - \lambda) (\alpha - h(\theta))}{\phi(h(\theta) - \lambda) + \theta (\alpha - h(\theta))}(\phi - \theta) \]
for all $0 < \theta \leq \phi \leq 1$, 
then there exists a compact perfect set $F \subset \mathbb{R}^d$ such that $\dim_{\mathrm A} F = \alpha$, $\dim_{\mathrm L} F = \lambda$, and $\dim_{\theta} F = h(\theta)$ for all $\theta \in [0,1]$.

A consequence of Corollary~\ref{hardcomparisoncor} is the following relationships between the $\Phi$-intermediate and intermediate dimensions. %

\begin{prop}\label{compareintermediate}
Let $\Phi$ be any admissible function, and let
\begin{equation}\label{compareintermediatedefinitions} \theta_1 \coloneqq \liminf_{\delta \to 0^+} \frac{\log \delta}{\log \Phi(\delta)}; \qquad \theta_2 \coloneqq \limsup_{\delta \to 0^+} \frac{\log\delta}{\log\Phi(\delta)},
\end{equation}
 noting that $0\leq \theta_1 \leq \theta_2 \leq 1$. 
 If $\dim_{\mathrm A} F < \infty$ then the following bounds hold: 
 \begin{itemize}
\item If $0=\theta_2=\lim_{\delta \to 0^+} \frac{\log \delta}{\log \Phi(\delta)}$ then $\underline{\dim}^\Phi F \leq \underline{\dim}_\theta F$ and $\overline{\dim}^\Phi F \leq \overline{\dim}_\theta F$ for all $\theta \in (0,1]$. 

\item If $0=\theta_1<\theta_2$ then $\underline{\dim}_{\theta_2} F \leq \overline{\dim}^\Phi F \leq \overline{\dim}_{\theta_2} F$ (so if $\dim_{\theta_2} F$ exists then $\overline{\dim}^\Phi F = \dim_{\theta_2} F$), and $\underline{\dim}^\Phi F \leq \min\{\overline{\dim}_\theta F,\underline{\dim}_{\theta_2} F\}$ for all $\theta \in (0,1]$. 

\item If $0< \theta_1\leq\theta_2$ then 
\begin{gather*}
\underline{\dim}_{\theta_1} F \leq \underline{\dim}^\Phi F \leq \min\{\overline{\dim}_{\theta_1} F,\underline{\dim}_{\theta_2} F\}, \\
\max\{\overline{\dim}_{\theta_1} F,\underline{\dim}_{\theta_2} F\} \leq \overline{\dim}^\Phi F \leq \overline{\dim}_{\theta_2} F.
\end{gather*}

\item If $0<\theta_1=\theta_2$ then $\underline{\dim}^\Phi F = \underline{\dim}_{\theta_1} F$ and $\overline{\dim}^\Phi F =\overline{\dim}_{\theta_1}$. 
\end{itemize}
\end{prop}

\begin{proof}
As an example, we prove $\overline{\dim}^\Phi F \leq \overline{\dim}_{\theta_2} F$ under the assumption that $\theta_2 > 0$; the other bounds are proved similarly. If $\theta_2 = 1$ then this follows from Proposition~\ref{basicbounds}, so assume $\theta_2 \in (0,1)$. Then letting $\eta \in (0,1-\theta_2)$, by the definition of $\theta_2$, 
\[ \limsup_{\delta \to 0^+}\frac{\Phi(\delta)}{\delta^{1/(\theta_2 + \eta)}} = 0 < \infty. \]
Corollary~\ref{hardcomparisoncor}~(i) now gives $\overline{\dim}^\Phi F \leq \overline{\dim}_{\theta_2+\eta} F$. The intermediate dimensions are continuous at $\theta_2 > 0$ by Theorem~\ref{intermediatects} so the result follows upon letting $\eta \to 0^+$. 
\end{proof}

For sets whose upper intermediate dimensions are continuous at $\theta=0$, usually we will not study the $\Phi$-intermediate dimensions, because much information about the general $\Phi$-intermediate dimensions of such sets can be obtained directly from results about their intermediate dimensions and these inequalities.

\section{H\"older and Lipschitz maps}\label{holdersection}

\subsection{H\"older distortion}

We now investigate how these dimensions behave under H{\"o}lder and Lipschitz maps. 
We say that a map $f \colon X \to Y$ is \emph{H{\"o}lder}, \emph{$\alpha$-H{\"o}lder} or \emph{($C,\alpha$)-H{\"o}lder} if 
\[ d_Y(f(x_1),f(x_2)) \leq C d_X(x_1,x_2)^\alpha \qquad \mbox{for all }x_1,x_2 \in X\] 
for constants $\alpha \in (0,1]$ and $C \in [0,\infty)$, and we call $\alpha$ the \emph{exponent}. Interestingly, the familiar upper bound $\dim f(F) \leq \alpha^{-1}\dim F$ for the image of a `reasonable' set $F$ under an $\alpha$-H{\"o}lder map $f$, which holds for the Hausdorff, box and intermediate dimensions (see Corollary~\ref{holderintermediate}), is different to the bound that is obtained for the $\Phi$-intermediate dimensions in the main result of this section,~Theorem~\ref{holder}. 
Fraser~\cite{Fraser2019-4} uses the Assouad spectrum to give bounds on the possible H{\"o}lder exponents of maps from an interval to a natural class of spirals. 
These bounds are better than bounds that have been obtained using any other notion of dimensions. 
A possible direction for future research would be to give similar applications of results in this section to obtain information about the possible H{\"o}lder exponents of maps between sets.

\begin{thm}\label{holder}
Let $\Phi$ and $\Phi_1$ be admissible functions and let $(X,d_X)$ and $(Y,d_Y)$ be uniformly perfect. 
Let $f \colon F \to Y$ be a H{\"o}lder map with exponent $\alpha \in (0,1]$ for some $F \subseteq X$, assume $\dim_\mathrm{A} f(F) < \infty$, and let $\gamma \in [1,1/\alpha]$. 
Assume that 
\begin{equation}\label{phiholdercondition}
 \Phi_1(\delta) \leq (\Phi(\delta^{1/(\alpha \gamma)}))^\alpha
 \end{equation}%
 for all sufficiently small $\delta$, and suppose $\overline{\dim}^\Phi F < \alpha \dim_\mathrm{A} f(F)$. 
  Then 
\[ \overline{\dim}^{\Phi_1} f(F) \leq \frac{\overline{\dim}^\Phi F + \alpha (\gamma - 1)\dim_\mathrm{A} f(F)}{\alpha \gamma}.\]
The same holds with $\overline{\dim}$ replaced by $\underline{\dim}$ throughout. 
\end{thm}

\begin{proof}
The idea of the proof is to consider a cover of $F$ with diameters in $[\Phi(\delta),\delta]$, consider the cover of $f(F)$ formed by the images under $f$ of this cover, and `fatten' the smallest sets in the new cover to size $\Phi_1(\delta^{\alpha \gamma})$ and break up the largest sets in the new cover to size $\delta^{\alpha \gamma}$. 
Assume that $f$ is ($C,\alpha$)-H{\"o}lder with $C \geq 1$. Let $\epsilon > 0$. Let  
\[ t > \frac{\overline{\dim}^\Phi F + \alpha (\gamma - 1)\dim_\mathrm{A} f(F)}{\alpha \gamma}. \]
Then there exist $s > \overline{\dim}^\Phi F$ and $a> \dim_\mathrm{A} f(F)$ such that $s < \alpha a$ and $t > (s + \alpha (\gamma - 1)a)/(\alpha \gamma)$. 
 Define  
\[g(\eta) \coloneqq \frac{\eta s + \alpha a (\gamma - \eta )}{\alpha \gamma}. \] 
 Since $a> \dim_\mathrm{A} f(F)$, there exists $M \in \mathbb{N}$ such that $N_r(B(y,R)\cap f(F)) \leq M(R/r)^a$ for all $y \in f(F)$ and $0<r<R$. 
Let $c \in (0,1)$ be such that $X$ and $Y$ are $c$-uniformly perfect. 
For all small enough $\delta$ we have $\Phi(\delta)/\delta < c/2$ and $\Phi_1(\delta)/\delta < c/2$, and there exists a cover $\{U_i\}$ of $F$ such that $\Phi(\delta) \leq |U_i| \leq \delta$ for all $i$, and 
\begin{equation}\label{holdersumbound}
\sum_i |U_i|^s \leq ((C + c^{-1})^{s/\alpha} + M(2C)^{a + \gamma g(1)})^{-1} \epsilon/2.
\end{equation}
Without loss of generality assume $U_i \cap F \neq \varnothing$ for all $i$. Now, $\{f(U_i)\}$ covers $f(F)$, and $|f(U_i)| \leq C|U_i|^\alpha$ for all $i$. 
There are two cases. 

\textbf{Case 1}: Suppose $i$ is such that $|f(U_i)| \leq \delta^{\alpha \gamma}/2$. Fix any $y_i \in f(U_i)$. There exists $y_i' \in Y$ such that $\Phi_1(\delta^{\alpha \gamma}) \leq d_Y(y_i,y_i') \leq \Phi_1(\delta^{\alpha \gamma})/c$, hence $d_Y(y_i,y_i') \leq (\Phi(\delta))^\alpha/c$. Let $V_i \coloneqq f(U_i) \cup \{y_i'\}$. 
By the triangle inequality, %
\begin{equation}\label{holderdiambound1} 
\Phi_1(\delta^{\alpha \gamma}) \leq d_Y(y_i,y_i') \leq |V_i| \leq |f(U_i)| + \Phi_1(\delta^{\alpha \gamma})/c \leq \delta^{\alpha \gamma}.
\end{equation}
Moreover, by the assumption~\eqref{phiholdercondition} about $\Phi_1$, 
\begin{equation}\label{holderubound1} |V_i| \leq |f(U_i)| + \Phi_1(\delta^{\alpha \gamma})/c \leq C|U_i|^\alpha + (\Phi(\delta))^\alpha/c \leq (C+ c^{-1}) |U_i|^\alpha. 
\end{equation}

\textbf{Case 2}: Now suppose that $i$ is such that $\delta^{\alpha \gamma}/2 < |f(U_i)| \leq C\delta^\alpha$. Then $(2C)^{-1/\alpha}\delta^\gamma < |U_i| \leq \delta$ so there exists $\beta_i \in [1,\gamma]$ such that $(2C)^{-1/\alpha} \delta^{\beta_i} < |U_i| \leq \delta^{\beta_i}$. 
Then $\delta^{\alpha \gamma}/2 < |f(U_i)| \leq C\delta^{\alpha \beta_i} \leq C\delta^\alpha$. %
There exists a collection of 
\[ M(2C)^a \delta^{\alpha (\beta_i - \gamma)a} \leq M(2C)^a |U_i|^{\alpha a (1-\gamma/\beta_i)} \]  %
or fewer balls, each of diameter at most $\delta^{\alpha \gamma}/2$, which cover $f(U_i) \cap f(F)$. For each ball we can add a point in $Y$ whose distance from the centre of the ball is between $\Phi_1(\delta^{\alpha \gamma})$ and $\Phi_1(\delta^{\alpha \gamma})/c$. Each of the new sets, which we call $\{W_{i,j}\}_j$, will satisfy 
\begin{equation}\label{holderdiambound2}
 \Phi_1(\delta^{\alpha \gamma}) \leq |W_{i,j}| \leq  \delta^{\alpha \gamma}.
 \end{equation}
Moreover, 
\begin{equation}\label{holderubound2}
|W_{i,j}| \leq \delta^{\alpha \gamma} =  (2C)^{\gamma/\beta_i} ((2C)^{-1/\alpha}\delta^{\beta_i})^{\alpha \gamma/\beta_i} \leq (2C)^{\gamma/\beta_i} |U_i|^{\alpha \gamma/\beta_i}. 
\end{equation}
Note that $g(\eta)$ is linear and decreasing in $\eta$, so $t>g(1) \geq g(\eta) \geq g(\gamma) = s/\alpha$ for all $\eta \in [1,\gamma]$, and in particular $t > g(\beta_i)$ for all $i$. Therefore using~\eqref{holderubound1} and~\eqref{holderubound2},  
\begin{align*}
\sum_k |V_k|^t + &\sum_{i,j} |W_{i,j}|^t < \sum_k |V_k|^{s/\alpha} + \sum_{i,j} |W_{i,j}|^{g(\beta_i)} \\
&\leq \sum_k ((C + c^{-1})|U_k|^\alpha)^{s/\alpha} + \sum_i M(2C)^a |U_i|^{\alpha a (1-\gamma/\beta_i)} ((2C)^{\gamma/\beta_i} |U_i|^{\alpha \gamma/\beta_i})^{g(\beta_i)} \\
&\leq (C + c^{-1})^{s/\alpha} \sum_k |U_k|^s + M(2C)^{a + \gamma g(\beta_i)/\beta_i} \sum_i |U_i|^s \\*
&\leq \epsilon,
\end{align*}
where the last equality follows from~\eqref{holdersumbound}. 
Also, $\{V_k\}_k \cup \{W_{i,j}\}_{i,j}$ covers $f(F)$, and noting~\eqref{holderdiambound1} and~\eqref{holderdiambound2}, we have $\overline{\dim}^{\Phi_1} f(F) \leq t$, as required. 
\end{proof}

We make several comments about Theorem~\ref{holder}. 
\begin{itemize}
\item An important special case is when $\gamma = 1/\alpha$ and $\Phi_1 = \Phi$. Then we can conclude 
\[ \overline{\dim}^\Phi f(F) \leq \overline{\dim}^\Phi F + (1-\alpha)\dim_\mathrm{A} f(F). \]
\item Another special case is for the $\Phi_1$ which satisfy~\eqref{phiholdercondition} with $\gamma = 1$, when we can conclude $\overline{\dim}^{\Phi_1} f(F) \leq \alpha^{-1}\overline{\dim}^\Phi F$. %
\item If $\overline{\dim}^\Phi F \geq \alpha \dim_\mathrm{A} f(F)$ (contrary to the assumption of Theorem~\ref{holder}) then the simple bound $\overline{\dim}^\Phi f(F) \leq \alpha^{-1} \overline{\dim}^\Phi F$ follows immediately. %
\item If $\overline{\dim}^\Phi F < \alpha \dim_\mathrm{A} f(F)$ but we only assume  that~\eqref{phiholdercondition} holds along a subsequence of $\delta \to 0^+$, then we can conclude only that 
\[ \underline{\dim}^{\Phi_1} f(F) \leq \frac{\overline{\dim}^\Phi F + \alpha (\gamma - 1)\dim_\mathrm{A} f(F)}{\alpha \gamma}.\]
\end{itemize}

Setting $\Phi(\delta) = \delta^{1/\theta}$ gives a H{\"o}lder distortion estimate for the intermediate dimensions in Corollary~\ref{holderintermediate}. 
For subsets of Euclidean space, Corollary~\ref{holderintermediate} was noted in~\cite[Section~14.2.1 5.]{Falconer2021-2}, and it also follows from the stronger result~\cite[Theorem~3.1]{Burrell2022brownian} which is proven using capacity theoretic methods and dimension profiles, but we include it nonetheless because our proof works for more general metric spaces. 

\begin{cor}\label{holderintermediate}
If $f \colon F \to Y$ is an $\alpha$-H{\"o}lder map with exponent $\alpha \in (0,1]$ and $\dim_\mathrm{A} f(F) < \infty$, then $\overline{\dim}_\theta f(F) \leq \alpha^{-1}\overline{\dim}_\theta F$ and $\underline{\dim}_\theta f(F) \leq \alpha^{-1}\underline{\dim}_\theta F$ for all $\theta \in [0,1]$. 
\end{cor}

\begin{proof} 
These estimates hold for the Hausdorff and lower and upper box dimensions (similar to~\cite[Exercise~2.2 and Proposition~3.3]{Falconer2014}), so assume that $\theta \in (0,1)$ and let $\Phi(\delta) = \Phi_1(\delta) = \delta^{1/\theta}$. 
If $\overline{\dim}_\theta F \geq \alpha \dim_\mathrm{A} f(F)$ then $\overline{\dim}_\theta f(F) \leq \dim_\mathrm{A} f(F) \leq \alpha^{-1} \overline{\dim}_\theta F$. If $\overline{\dim}_\theta F < \alpha \dim_\mathrm{A} f(F)$ then since 
\[ \Phi_1(\delta) = \Phi(\delta) = \delta^{1/\theta} = ((\delta^{1/\alpha})^{1/\theta})^\alpha = \Phi(\delta^{1/\alpha})^\alpha, \]
the case $\gamma = 1$ of Theorem~\ref{holder} gives that $\overline{\dim}_\theta f(F) \leq \alpha^{-1} \overline{\dim}_\theta F$. 
Similarly, the bound for the lower intermediate dimensions follows from the version of Theorem~\ref{holder} for the lower $\Phi$-intermediate dimensions. %
\end{proof}

\subsection{Lipschitz stability}

Recall that a map is \emph{Lipschitz} if it is 1-H\"older, and \emph{bi-Lipschitz} if it is Lipschitz with a Lipschitz inverse. 
Corollary~\ref{philipschitz} shows that the $\Phi$-intermediate dimensions cannot increase under Lipschitz maps. %
We also show that $\overline{\dim}^\Phi$ and $\underline{\dim}^\Phi$ are stable under bi-Lipschitz maps, which is an important property that most notions of dimension satisfy. This shows that the $\Phi$-intermediate dimensions provide further invariants for the classification of sets up to bi-Lipschitz image. 
Bi-Lipschitz stability has already been proven for the Hausdorff and box dimensions in~\cite[Propositions~2.5 and 3.3]{Falconer2014} and, for subsets of $\mathbb{R}^n$, for the intermediate dimensions in~\cite[Lemma~3.1]{Fraser2021-1}.

\begin{cor}\label{philipschitz}
Let $X$ and $Y$ be underlying spaces, let $F \subseteq X$, let $f \colon F \to Y$ be Lipschitz, and assume that $\dim_\mathrm{A} f(F) < \infty$. Then 
\begin{enumerate}
\item We have $\overline{\dim}^\Phi f(F) \leq \overline{\dim}^\Phi F$ and $\underline{\dim}^\Phi f(F) \leq \underline{\dim}^\Phi F$. 

\item\label{bilipschitzlabel} If moreover $f$ is bi-Lipschitz then 
$\overline{\dim}^\Phi f(F) = \overline{\dim}^\Phi F$
and
$\underline{\dim}^\Phi f(F) = \underline{\dim}^\Phi F$. 
\end{enumerate}
\end{cor}
In~\ref{bilipschitzlabel}, the assumption $\dim_\mathrm{A} f(F) < \infty$ is equivalent to $\dim_\mathrm{A} F < \infty$ since the Assouad dimension is stable under bi-Lipschitz maps. 

\begin{proof} 
If $\overline{\dim}^\Phi F \geq \dim_\mathrm{A} f(F)$ then $\overline{\dim}^\Phi f(F) \leq \dim_\mathrm{A} f(F) \leq \overline{\dim}^\Phi F$ by Proposition~\ref{basicbounds}; if $\overline{\dim}^\Phi F < \dim_\mathrm{A} f(F)$ then the case $\alpha = \gamma = 1$, $\Phi_1 = \Phi$, of Theorem~\ref{holder} gives $\overline{\dim}^\Phi f(F) \leq \overline{\dim}^\Phi F$. The proof that $\underline{\dim}^\Phi f(F) \leq \underline{\dim}^\Phi F$ is similar, and 2. follows from 1. %
\end{proof}

\section{A mass distribution principle}\label{masssection}

In this section we prove a mass distribution principle for the $\Phi$-intermediate dimensions and a converse result (a Frostman type lemma), which together give an alternative characterisation of the intermediate dimensions. We then prove some applications regarding product sets and finite stability. 

\subsection{A mass distribution principle}

The mass distribution principle is a useful tool to bound dimensions from below by putting a measure on the set. The original version was for the Hausdorff dimension (see~\cite[page~67]{Falconer2014}), and a version was proved for the intermediate dimensions in~\cite[Proposition~2.2]{Falconer2020}. The following natural generalisation for the $\Phi$-intermediate dimensions holds. 

\begin{lemma}\label{massdistprinc}
 Let $F$ be a subset and let $s,a,c,\delta_0>0$ be positive constants. 
 \begin{enumerate}[label=(\roman*)]
\item If there exists a positive decreasing sequence $\delta_n \to 0$ such that for each $n \in \mathbb{N}$ there exists a Borel measure $\mu_n$ with support $\mathrm{supp}(\mu_n) \subseteq F$ with $\mu_n (\mathrm{supp}(\mu_n)) \geq a$, and such that for every Borel subset $U \subseteq X$ with $\Phi(\delta_n) \leq |U| \leq \delta_n$ we have $\mu_n(U)\leq c |U|^s$, then $\overline{\dim}^\Phi F \geq s$. 

\item If, moreover, for all $\delta \in (0,\delta_0)$ there exists a Borel measure $\mu_\delta$ with support $\mathrm{supp}(\mu_\delta) \subseteq F$ with $\mu_\delta (\mathrm{supp}(\mu_\delta)) \geq a$, and such that for every Borel subset $U \subseteq X$ with $\Phi(\delta) \leq |U| \leq \delta$ we have $\mu_\delta(U)\leq c |U|^s$, then $\underline{\dim}^\Phi F \geq s$. 
\end{enumerate}
\end{lemma}

\begin{proof}
We prove (i); the proof of~(ii) is similar. If $n \in \mathbb{N}$ and $\{U_i\}$ is a cover of $F$ such that $\Phi(\delta_n) \leq |U_i| \leq \delta_n$ for all $i$, then the closures $\overline{U_i}$ are Borel, satisfy $\Phi(\delta_n) \leq |\overline{U_i}| = |U_i| \leq \delta_n$, and cover $\mathrm{supp}(\mu_n)$, so
\begin{equation}\label{massdistprinceqn} 
a \leq \mu_n (\mathrm{supp}(\mu_n)) = \mu_n\left(\bigcup_i \overline{U_i}\right) \leq \sum_i \mu_n(\overline{U_i}) \leq c\sum_i |\overline{U_i}|^s = c\sum_i |U_i|^s. 
\end{equation}
 Therefore $\sum_i |U_i|^s \geq a/c > 0$, so $\overline{\dim}^\Phi F \geq s$. 
 \end{proof}
 
 \subsection{A Frostman type lemma}
 
Another powerful tool in fractal geometry and geometric measure theory is Frostman's lemma, dual to the mass distribution principle. 
The following analogue of Frostman's lemma for the $\Phi$-intermediate dimensions holds, generalising~\cite[Proposition~2.3]{Falconer2020} for the intermediate dimensions both to more general functions~$\Phi$ and to more general metric spaces. 
In the proof, we use notation from~\cite[Theorem~2.2]{Hytonen2010}, where $\delta$ denotes a certain constant. 

\begin{lemma}\label{frostman}
Assume that $\dim_\mathrm{A} F < \infty$. 
\begin{enumerate}[label=(\roman*)]
\item If $\overline{\dim}^\Phi F > 0$ then for all $s \in (0,\overline{\dim}^\Phi F)$ there exists a constant $c \in (0,\infty)$ such that for all $\delta_0 >0$ there exist $\delta' \in (0,\delta_0)$ and a Borel probability measure $\mu_{\delta'}$ with finite support $\mathrm{supp}(\mu_{\delta'}) \subseteq F$ 
such that if $x \in X$ and $\Phi(\delta') \leq r \leq \delta'$ then 
\[\mu_{\delta'} (B(x,r)) \leq c r^s.\]

\item If $\underline{\dim}^\Phi F > 0$ then for all $s \in (0,\underline{\dim}^\Phi F)$ there exists $c \in (0,\infty)$ such that for all sufficiently small $\delta'$ there exists a Borel probability measure $\mu_{\delta'}$ with finite support $\mathrm{supp}(\mu_{\delta'}) \subseteq F$ such that if $x \in X$ and $\Phi(\delta') \leq r \leq \delta'$ then $\mu_{\delta'} (B(x,r)) \leq c r^s$. 
\end{enumerate}
\end{lemma}%

\begin{proof}
We prove~(ii); the proof of~(i) is similar. 
The idea of the proof is to put point masses on an analogue of dyadic cubes of size approximately $\Phi(\delta')$ so that the measure of sets with diameter approximately $\Phi(\delta')$ is controlled by the $\Phi$-intermediate dimension of $F$, and then iteratively reduce the masses so that the mass of larger cubes is not too large either. 
The proof is based on the proof of~\cite[Proposition~2.3]{Falconer2020} for the intermediate dimensions, which is in turn based on~\cite[pages~112--114]{Mattila1995}. In~\cite[Proposition~2.3]{Falconer2020}, the assumption that the set $F$ is closed is not necessary as it is not used in the proof. 

The main difference with the proof of~\cite[Proposition~2.3]{Falconer2020} is that in $\mathbb{R}^n$ there are the dyadic cubes to work with, but here we use the fact that $\dim_\mathrm{A} F < \infty$, and use an analogue of the dyadic cubes constructed in~\cite{Hytonen2010} for general doubling metric spaces. 
We now state a special case of \cite[Theorem~2.2]{Hytonen2010}, using notation from that theorem. 
We take the quasi-metric $\rho$ simply to be the metric $d$ restricted to $F$ (so the usual triangle inequality holds and $A_0 = 1$). 
Fix $\delta \coloneqq 1/20$ (in fact any $\delta \in (0,1/12)$ will do). 
Since $\dim_\mathrm{A} F < \infty$, for each $k \in \mathbb{N}$ we have $N_{\delta^k/3}(F) < \infty$. Therefore there exists a finite $\delta^k$-separated subset $\{z_{\alpha}^k\}_\alpha$ of $F$, of maximum possible cardinality. 
Then applying~\cite[Theorem~2.2]{Hytonen2010} with $c_0 = C_0 = 1$, $c_1 = 1/3$, $C_1 = 2$, for each $k \in \mathbb{N}$ there exist subsets $Q^k \coloneqq \{Q^k_\alpha\}_\alpha$ of $F$ such that: 
\begin{enumerate}
\item for all $k \in \mathbb{N}$, $F = \bigcup_\alpha Q_\alpha^k$ with the union disjoint;
\item\label{frostmanballinball} 
$B^F (z_\alpha^k,(20)^{-k}/4) \subseteq B^F (z_\alpha^k,c_1 (20)^{-k}) \subseteq Q_\alpha^k 
\subseteq B^F (z_\alpha^k,C_1 (20)^{-k})
= B^F (z_\alpha^k,2 (20)^{-k})$, recalling that $B^F$ denotes the open ball in $F$;
\item if $k,l \in \mathbb{N}$ with $k \leq l$ then for all $\alpha, \beta$, either $Q_\alpha^k \cap Q_\beta^l = \varnothing$ or $Q_\beta^l \subseteq Q_\alpha^k$, and in the latter case, also $B^F (z_\beta^l,2 (20)^{-l}) \subseteq B^F (z_\alpha^k,2 (20)^{-k})$. We call $Q_\alpha^k$ a \emph{parent} of $Q_\beta^l$. 
\end{enumerate}
We say that $Q_\alpha^k$ is a \emph{dyadic cube} with \emph{centre} $z_\alpha^k$. 

Let $c_2 \in (0,1)$ be such that $X$ is $c_2$-uniformly perfect. 
Suppose $\underline{\dim}^\Phi F > 0$ and let $s \in (0,\underline{\dim}^\Phi F)$. Then there exists $\epsilon > 0$ such that for all sufficiently small $\delta'$ and all covers $\{U_i\}$ of $F$ satisfying $\Phi(\delta') \leq |U_i| \leq \delta'$ for all $i$, %
\begin{equation}\label{frostmansumlarge}
\sum_i |U_i|^s \geq \epsilon.
\end{equation} 
Let $\delta'$ be small enough such that this is the case, and moreover that $\Phi(\delta')/\delta' < c_2/320$. Define $m = m(\delta')$ to be the largest natural number satisfying $\Phi(\delta') \leq \frac{1}{2}(20)^{-m}$. Define the Borel measure $\mu_m$ by 
\[ \mu_m \coloneqq \sum_{\alpha} 20^{-ms} M_{z_\alpha^k} \]
where $M_{z_\alpha^k}$ is a unit point mass at $z_\alpha^k$. 

Let $l$ be the largest integer such that $8(20^{-(m-l)}) \leq \delta'$, noting that $l \geq 1$. In particular, $|Q_{m-l}| \leq \delta'/2$ for all $Q_{m-l} \in Q^{m-l}$. 
In order to reduce the mass of cubes which carry too much measure, having defined $\mu_{m-k}$ for some $k \in \{0,1,\dotsc,l-1\}$, inductively define the Borel measure $\mu_{m-k-1}$, supported on the same finite set as $\mu_m$, by 
\[ \mu_{m-k-1} |_{Q_{m-k-1}} \coloneqq \min\left\{1, \frac{20^{-(m-k-1) s}}{ \mu_{m-k}(Q_{m-k-1})}\right\}\mu_{m-k}|_{Q_{m-k-1}} \]
for all $Q_{m-k-1} \in Q^{m-k-1}$. By construction, if $k \in \{0,1,\dotsc, l\}$ and $Q_{m-k} \in Q^{m-k}$ then 
\begin{equation}\label{frostmanupperbound} \mu_{m-l}(Q_{m-k}) \leq 20^{-(m-k)s} \leq 4^s c_2^{-s} |Q_{m-k}|^s 
\end{equation}
by condition~\ref{frostmanballinball}. %
Moreover, each $Q_m \in Q^m$ satisfies $\mu_m(Q_m) = 20^{-ms}$. If $k \in \{0,1,\dotsc, l-1\}$ and $Q_{m-k} \in Q^{m-k}$ satisfies $\mu_{m-k}(Q_{m-k}) = 20^{-(m-k)s}$ and $Q_{m-k-1} \in Q^{m-k-1}$ is the parent of $Q_{m-k}$, then by the construction of $\mu_{m-k-1}$, either $\mu_{m-k-1}(Q_{m-k}) = 20^{-(m-k)s}$ or $\mu_{m-k-1}(Q_{m-k-1}) = 20^{-(m-k-1)s}$. Therefore for all $y \in F$ there is at least one $k \in \{0,1,\dotsc, l\}$ and $Q_y \in Q^{m-k}$ with $y \in Q_y$ such that 
\begin{equation}\label{frostmanmulowerbound}
\mu_{m-l}(Q_y) = 20^{-(m-k)s} \geq 4^{-s}|Q_y|^s, 
\end{equation}
where the inequality is by condition~\ref{frostmanballinball}. 

 For each $y \in F$, choosing $Q_y$ such that~\eqref{frostmanmulowerbound} is satisfied and moreover $Q_y \in Q^{m-k}$ for the largest possible $k \in \{0,1,\dotsc, l\}$ yields a finite collection of cubes $\{Q_i\}$ which cover $F$. For each $i$, let $z_i$ be the centre of $Q_i$, and by the uniformly perfect condition there exists $p_i \in X$ such that $\Phi(\delta') \leq d(p_i,z_i) \leq \Phi(\delta')/c_2 \leq \delta'/2$. Letting $U_i \coloneqq Q_i \cup \{p_i\}$, by condition~\ref{frostmanballinball} we have $\Phi(\delta') \leq |U_i| \leq \delta'$. Then $\{U_i\}$ covers $F$, and each $|U_i| \leq |Q_i| + \Phi(\delta')/c_2 \leq (1+1/c_2)|Q_i|$. 
 Therefore by~\eqref{frostmansumlarge} and~\eqref{frostmanmulowerbound},  
\begin{equation}\label{frostmangreaterepsilon}
\mu_{m-l}(F) = \sum_i \mu_{m-l}(Q_i) \geq \sum_i 4^{-s}|Q_i|^s \geq 4^{-s}(1+1/c_2)^{-s} \sum_i |U_i|^s \geq 4^{-s}(1+1/c_2)^{-s} \epsilon. 
\end{equation}
Define $\mu_{\delta'} \coloneqq (\mu_{m-l}(F))^{-1} \mu_{m-l}$, which is clearly a Borel probability measure with finite support $\mathrm{supp}(\mu_{\delta'}) \subseteq F$. 

Now, since $\dim_\mathrm{A} F < \infty$ there exists $C \in \mathbb{N}$ such that $N_d(B^F(p,13d)) \leq C$ for all $p \in F$ and $d>0$. 
Let $x \in X$ and $r \in [\Phi(\delta'), \delta']$. 
Let $j = j(r)$ be the largest integer in $\{0,1,\dotsc,l\}$ such that $20^{-(m-j+1)} < r$; such an integer exists by the definition of $m$. 
If $B^X (x,r) \cap F = \varnothing$ then $\mu_{\delta'}(B^X (x,r)) = 0$, so suppose that there exists some $x_1 \in B^X (x,r) \cap F$, so $B^X (x,r) \subseteq B^F(x_1,2r)$. 
Suppose $B^X (x,r) \cap Q_{m-j} \neq \varnothing$ for some $Q_{m-j} \in Q^{m-j}$, with centre $z_{m-j}$, say. 
Then there exists $z \in B^X(x,r) \cap Q_{m-j}$, and by condition~\ref{frostmanballinball} and the definition of $j$,  
\[ d(x_1,z_{m-j}) \leq d(x_1,z) + d(z,z_{m-j}) \leq 2r + 2(20)^{-(m-j)} \leq 6(20)^{-(m-j)}. \]
Therefore $z_{m-j} \in B^F (x_1,6(20)^{-(m-j)})$, and the centres of the cubes in $Q_{m-j}$ which intersect $B^X (x,r)$ form a $20^{-(m-j)}$-separated subset of $B^F (x_1,6(20)^{-(m-j)})$. But 
\[ N_{6(20)^{-(m-j)}/13}(B^F (x_1,6(20)^{-(m-j)})) \leq C.\] 
Therefore there are most $C$ such centres, so at most $C$ elements of $Q^{m-j}$ which intersect $B^X(x,r)$. Therefore by~\eqref{frostmanupperbound} and~\eqref{frostmangreaterepsilon} and the definition of $j$,  
\[ \mu_{\delta'} (B^X(x,r)) = (\mu_{m-l}(F))^{-1} \mu_{m-l}(B^X(x,r)) \leq C (\mu_{m-l}(F))^{-1} 20^{-(m-j)s} \leq c r^s, \]
where $c \coloneqq C 4^s (1+1/c_2)^s \epsilon^{-1} (20)^s$, as required. 
\end{proof}

Putting Lemmas~\ref{massdistprinc} and~\ref{frostman} together, we obtain a useful characterisation of the $\Phi$-intermediate dimensions. 

\begin{thm}\label{massfrostman}
If $\Phi$ is an admissible function and $\dim_\mathrm{A} F < \infty$ then 
\begin{align*} (i) \quad \overline{\dim}^\Phi F = \sup\{ s \geq 0 : &\mbox{ there exists } C \in (0,\infty) \mbox{ such that for all } \delta_1 > 0 \\
&\mbox{ there exists } \delta \in (0,\delta_1) \mbox{ and a Borel probability measure } \mu_\delta  \\
&\mbox{ with support } \mathrm{supp}(\mu_\delta) \subseteq F \mbox{ such that if } U \mbox{ is a Borel subset of } \\
& X \mbox{ which satisfies } \Phi(\delta) \leq |U| \leq \delta \mbox{ then } \mu_\delta(U) \leq C|U|^s \} 
\end{align*}%

\begin{align*} (ii) \quad \underline{\dim}^\Phi F = \sup\{ s \geq 0 : &\mbox{ there exist } C,\delta_1 \in (0,\infty) \mbox{ such that for all } \delta \in (0,\delta_1) \mbox{ there }\\ 
&\mbox{ exists a Borel probability measure } \mu_\delta \mbox{ with support } \\
&\mathrm{supp}(\mu_\delta) \subseteq F \mbox{ such that if } U \mbox{ is a Borel subset satisfying } \\
&\Phi(\delta) \leq |U| \leq \delta \mbox{ then } \mu_\delta(U) \leq C|U|^s \} 
\end{align*}
\end{thm}

\begin{proof}
We prove~(ii) using Lemma~\ref{massdistprinc}~(ii) and Lemma~\ref{frostman}~(ii); (i) follows from Lemma~\ref{massdistprinc}~(i) and Lemma~\ref{frostman}~(i) in a similar way. 
We denote by $\mathit{sup}$ the supremum on the right-hand side of the equation~(ii). 
Fix $y \in F$. 
If $s=0$, then letting $C\coloneqq 1$ and letting $\mu_\delta$ be a unit point mass at $y$ for all sufficiently small $\delta$, we see that $\mathit{sup}$ is well-defined and non-negative. Suppose that $\underline{\dim}^\Phi F > 0$ and let $s \in (0,\underline{\dim}^\Phi F)$. Then by the Frostman type Lemma~\ref{frostman}~(ii), there exist constants $c,\delta_1 \in (0,\infty)$ such that for all $\delta \in (0,\delta_1)$ there exists a Borel probability measure $\mu_{\delta}$ with finite support $\mathrm{supp}(\mu_{\delta}) \subseteq F$ such that if $x \in X$ and $\Phi(\delta) \leq r \leq \delta$ then $\mu_{\delta} (B(x,r)) \leq c r^s$. 
If $U$ is a Borel subset of $X$ satisfying $\Phi(\delta) \leq |U| \leq \delta$, then $U \cap F = \varnothing$ implies $\mu_\delta(U) = 0$. Suppose there exists some $x \in U \cap F$. Let $M$ be the doubling constant of $F$. Then $U \cap \mathrm{supp}(\mu_\delta) \subseteq B(x,2|U|)$, so there exist $x_1,\dotsc,x_M \in B^F(x,2|U|)$ such that $U \cap \mathrm{supp}(\mu_\delta) \subseteq B^F(x,2|U|) \subseteq \bigcup_{i=1}^M B^F(x_i,|U|)$. Therefore 
\[\mu_\delta(U) \leq \sum_{i=1}^M \mu_\delta (B^F(x_i,|U|)) = \sum_{i=1}^M \mu_\delta(B^X(x_i,|U|)) \leq C|U|^s,\]
where $C \coloneqq Mc$. %
Thus $s \leq \mathit{sup}$. 

For the reverse inequality, if $\mathit{sup}> 0$ and $t \in (0, \mathit{sup})$ then by the mass distribution principle Lemma~\ref{massdistprinc}~(ii), $t \leq \underline{\dim}^\Phi F$. 
Therefore if $\max\{\mathit{sup}, \underline{\dim}^\Phi F\} > 0$ then in fact $\mathit{sup} = \underline{\dim}^\Phi F$. But both $\mathit{sup}$ and $\underline{\dim}^\Phi F$ are non-negative, so they must always be equal. 
\end{proof}

\subsection{Product formulae}

It is a well-studied problem to bound the dimensions of product sets in terms of the dimensions of the marginals. Very often, dimensions come in pairs (dim, Dim) which satisfy $\dim F \leq \mbox{Dim} F$ and 
\begin{equation}\label{dimpairineq} \dim E + \dim F \leq \dim (E \times F) \leq \dim E + \mbox{Dim} F \leq \mbox{Dim} (E \times F) \leq \mbox{Dim} E + \mbox{Dim} F \end{equation}
for all `reasonable' sets $E$ and $F$ and `reasonable' metrics on the product space. Examples are (Hausdorff, packing)~\cite{Howroyd1996}, (lower box, upper box)~\cite{Robinson2013}, (lower, Assouad) and (modified lower, Assouad)~\cite[Corollary~10.1.2]{Fraser2020} and, for any fixed $\theta \in (0,1)$, (lower spectrum at $\theta$, Assouad spectrum at $\theta$) and (modified lower spectrum at $\theta$, Assouad spectrum at $\theta$)~\cite[Proposition~4.4]{Fraser2018-2}. In Theorem~\ref{producttheorem} we show that for any given $\Phi$ or $\theta$, (lower $\Phi$-intermediate, upper box) and (lower $\theta$-intermediate, upper box) are also such pairs. However, our upper bound for $\overline{\dim}^\Phi (E \times F)$ is $\overline{\dim}^\Phi E + \overline{\dim}_\mathrm{B} F$, rather than the expected $\overline{\dim}^\Phi E + \overline{\dim}^\Phi F$. Theorem~\ref{producttheorem} generalises~\cite[Proposition~2.5]{Falconer2020} on the intermediate dimensions of product sets to more general functions $\Phi$ and more general metric spaces, and also gives an improved lower bound for $\overline{\dim}^\Phi (E \times F)$ and an improved upper bound for $\underline{\dim}^\Phi (E \times F)$. 
We improve the lower bound for $\overline{\dim}^\Phi (E \times F)$ further for self-products in~(iii). 
A possible direction for future research would be to investigate the sharpness of the bounds in Theorem~\ref{producttheorem}. 

\begin{thm}\label{producttheorem}
Consider uniformly perfect metric spaces $(X,d_X)$ and $(Y,d_Y)$. Let $d_{X \times Y}$ be a metric on $X \times Y$ such that there exist constants $c_1,c_2 \in (0,\infty)$ such that 
\begin{equation}\label{metriclikesup} 
c_1 \max(d_X,d_Y) \leq d_{X \times Y} \leq c_2 \max(d_X,d_Y). 
\end{equation}
Then if $E \subseteq X$ and $F \subseteq Y$ have finite Assouad dimension, then 
 \begin{enumerate}[label=(\roman*)]
 
\item $\overline{\dim}^\Phi E + \underline{\dim}^\Phi F 
\leq \overline{\dim}^\Phi (E \times F) 
\leq \overline{\dim}^\Phi E + \overline{\dim}_\mathrm{B} F$;  

\item $\underline{\dim}^\Phi E + \underline{\dim}^\Phi F 
\leq \underline{\dim}^\Phi (E \times F)
\leq \underline{\dim}^\Phi E + \overline{\dim}_\mathrm{B} F$. 

In the case of self-products,~(i) can be improved to %

\item 
 $2 \overline{\dim}^\Phi F \leq \overline{\dim}^\Phi (F \times F) \leq \overline{\dim}^\Phi F + \overline{\dim}_\mathrm{B} F$. 

\end{enumerate}

\end{thm}

Note that~\eqref{metriclikesup} is the same condition as~\cite[(2.4)]{Robinson2013}, and familiar metrics which satisfy this are $d_{X \times Y} \coloneqq \max\{d_X,d_Y\}$ and $d_{X \times Y} \coloneqq (d_X^p + d_Y^p)^{1/p}$ for $p \in [1,\infty)$. 

\begin{proof}
The idea of the proof of the upper bounds is to consider a cover of one of the sets $E$ with diameters in $[\Phi(\delta),\delta]$, and, for each set $U_i$ in that cover, to form a cover of that other set $F$ with all the diameters approximately equal to $|U_i|$, with the number of sets in this cover controlled by $\overline{\dim}_{\mathrm B} F$. We can then cover the product set with approximate squares with sizes between $\Phi(\delta)$ and $\delta$ to obtain the result. 
The idea of the proof of the lower bounds is to use the Frostman type lemma to put measures on each of the marginal sets such that the measure of sets with diameter in $[\Phi(\delta),\delta]$ is controlled by the $\Phi$-intermediate dimensions of the sets, and then apply the mass distribution principle with the product measure on the product set. 

Since $X$ and $Y$ each have more than one point, so does $X \times Y$. 
A straightforward calculation shows that since $(X,d_X)$ is uniformly perfect, so is $(X \times Y,d_{X \times Y})$. 
Another routine calculation shows that since $E$ and $F$ have finite Assouad dimension, so does $E \times F$. 

(i) We first prove the upper bound of~(i), following the proof of the upper bound in~\cite[Proposition~2.5]{Falconer2020}. Let $\epsilon > 0$. Let $c_p \in (0,1)$ be such that $X \times Y$ is $c_p$-uniformly perfect, and without loss of generality assume $0 < c_p < c_1 < 1 < c_2 < \infty$. 
Since $\dim_\mathrm{A} (E \times F) < \infty$ there exists $A \in \mathbb{N}$ such that $N_r(B^{E \times F} (p,4c_2 r)) \leq A$ for all $p \in E \times F$ and $r > 0$. 
 Let $\Delta > 0$ be such that $\Phi(\delta)/\delta < c_p/2$ for all $\delta \in (0,\Delta)$. 
Fix $s > \overline{\dim}^\Phi E$ and $d > \overline{\dim}_\mathrm{B} F$. Let $\delta_1\in (0,\Delta)$ be such that for all $r \in (0,\delta_1)$ there is a cover of $F$ by $r^{-d}$ or fewer sets, each having diameter at most $r$. Let $\delta_0 \in (0,\delta_1)$ be such that for all $\delta \in (0,\delta_0)$ there exists a cover $\{U_i\}$ of $E$ such that $\Phi(\delta) \leq |U_i| \leq \delta$ for all $i$, and 
\begin{equation}\label{productsetepsbound}
\sum_i |U_i|^s \leq A^{-1} (c_2 + c_p^{-1})^{-(s+d)}\epsilon. 
\end{equation}
For such a cover, for each $i$ let $\{U_{i,j}\}_j$ be a cover of $F$ by $|U_i|^{-d}$ or fewer sets, each having diameter $|U_{i,j}| \leq |U_i|$. 
Then for all $i$ and $j$,  
\begin{equation}\label{productbound}
|U_i \times U_{i,j}| \leq c_2 \max\{|U_i|,|U_{i,j}|\} = c_2|U_i| \leq c_2\delta, 
\end{equation}
so $U_i \times U_{i,j}$ can be covered by $A$ sets $\{U_{i,j,k}\}_k$, each having diameter at most 
\[ \min\{\delta/2,{|U_i \times U_{i,j}|}\}.\] 
We may assume that each of these sets is non-empty, and fix $p_{i,j,k} \in U_{i,j,k}$. Fix $q_{i,j,k} \in X \times Y$ such that $\Phi(\delta) \leq d_{X \times Y}(p_{i,j,k},q_{i,j,k}) \leq \Phi(\delta)/c_p$. Let $V_{i,j,k} \coloneqq U_{i,j,k} \cup \{q_{i,j,k}\}$, so by the triangle inequality 
\begin{equation}\label{productlimits}
 \Phi(\delta) \leq d_{X \times Y}(p_{i,j,k},q_{i,j,k}) \leq |V_{i,j,k}| \leq \delta/2 + \Phi(\delta)/c_p \leq \delta,
 \end{equation}
 since $\delta < \delta_0 < \Delta$. 
Also, by~\eqref{productbound}, 
\begin{equation}\label{productvbound} |V_{i,j,k}| \leq  c_2|U_i| + \Phi(\delta)/c_p \leq (c_2 + c_p^{-1})|U_i|. 
\end{equation}
Therefore by~\eqref{productvbound} and~\eqref{productsetepsbound},  
\[ \sum_{i,j,k} |V_{i,j,k}|^{s+d} \leq \sum_i A|U_i|^{-d} ((c_2 + c_p^{-1})|U_i|)^{s+d} \leq A (c_2 + c_p^{-1})^{s+d} \sum_i |U_i|^s \leq \epsilon.\]
Also \[E \times F \subseteq \bigcup_{i,j,k} U_{i,j,k} \subseteq \bigcup_{i,j,k} V_{i,j,k}. \]
This gives $\overline{\dim}^\Phi (E \times F) \leq s + d$, so $\overline{\dim}^\Phi (E \times F) \leq \overline{\dim}^\Phi E + \overline{\dim}_\mathrm{B} F$. The bound $\overline{\dim}^\Phi (E \times F) \leq \overline{\dim}_\mathrm{B} E + \overline{\dim}^\Phi F$ follows similarly. 

The proof of the lower bound is somewhat similar to the proof of the lower bound in~\cite[Proposition~2.5]{Falconer2020}. First assume $\underline{\dim}^\Phi F = 0$. Fix any $f \in F$. By~\eqref{metriclikesup}, the natural embedding $E \xhookrightarrow{} X \times Y$, $x \mapsto (x,f)$, is bi-Lipschitz onto its image, so by Corollary~\ref{philipschitz} 2. and Proposition~\ref{unprovedprop}~(i),
\[ \overline{\dim}^\Phi E + \underline{\dim}^\Phi F = \overline{\dim}^\Phi E = \overline{\dim}^\Phi (E \times \{f\}) \leq \overline{\dim}^\Phi (E \times F). \]
Now assume that $\overline{\dim}^\Phi E > 0$ and $\underline{\dim}^\Phi F > 0$. Fix $t_1 \in (0,\overline{\dim}^\Phi E)$ and $t_2 \in (0,\underline{\dim}^\Phi F)$. By Lemma~\ref{frostman}~(i) there exists $c_E \in (0,\infty)$ such that for all $\delta_2>0$ there exists $\delta_3 \in (0,\delta_2)$ and a Borel probability measure $\mu_{\delta_3}$ with $\mathrm{supp}(\mu_{\delta_3}) \subseteq E$ such that if $x \in X$ and $\Phi(\delta_3) \leq r_1 \leq \delta_3$ then $\mu_{\delta_3}(B^X(x,r_1)) \leq c_E r_1^{t_1}$. 
By Lemma~\ref{frostman}~(ii) there exist $c_F,\delta_4 \in (0,\infty)$ such that for all $\delta_5 \in (0,\delta_4)$ there exists a Borel probability measure $\nu_{\delta_5}$ with $\mathrm{supp}(\nu_{\delta_5}) \subseteq F$ such that if $y \in Y$ and $\Phi(\delta_5) \leq r_2 \leq \delta_5$ then $\nu_{\delta_5}(B^Y(y,r_2)) \leq c_F r_2^{t_2}$. 
If $\delta_7>0$, then there exists $\delta_6 \in (0,\min\{\delta_7,\delta_4\})$ and Borel probability measures $\mu_{\delta_6}$ and $\nu_{\delta_6}$ as above. Let $\mu_{\delta_6} \times \nu_{\delta_6}$ be the product measure, which satisfies $\mathrm{supp}(\mu_{\delta_6} \times \nu_{\delta_6}) \subseteq E \times F$. 

If $U \subseteq X \times Y$ is Borel and satisfies $\Phi(\delta_6) \leq |U| \leq \delta_6$ then if we fix any $(x,y) \in U$ then 
\begin{equation}\label{productlowerboundsubset} U \subseteq B^{X \times Y}((x,y), 2|U|) \subseteq B^X(x,2|U|/c_1) \times B^Y(y,2|U|/c_1). 
\end{equation}
Fix $x_1,\dotsc,x_C \in E$ and $y_1,\dotsc,y_C \in F$ such that 
\[ \overline{E} \cap B^X(x,2|U|/c_1) \subseteq \bigcup_{i=1}^C B^X(x_i,|U|); \qquad \overline{F} \cap B^Y(y,2|U|/c_1) \subseteq \bigcup_{i=1}^C B^Y(y_i,|U|). \]
Now, 
\begin{align*} (\overline{E} \times \overline{F}) \cap (B^X(x,2|U|/c_1) &\times B^Y(y,2|U|/c_1)) \\
&= (\overline{E} \cap B^X(x,2|U|/c_1)) \times (\overline{F} \cap B^Y(y,2|U|/c_1)) \\
&\subseteq \left(\bigcup_{i=1}^C B^X(x_i,|U|)\right) \times \left(\bigcup_{j=1}^C B^Y(y_j,|U|)\right) \\
&= \bigcup_{i=1}^C \bigcup_{j=1}^C (B^X(x_i,|U|) \times B^Y(y_j,|U|)).
\end{align*}
Then by~\eqref{productlowerboundsubset} and the definition of the product measure,  
\begin{align*}
 (\mu_{\delta_6} \times \nu_{\delta_6}) (U) &\leq (\mu_{\delta_6} \times \nu_{\delta_6}) (B^X(x,2|U|/c_1) \times B^Y(y,2|U|/c_1)) \\
 &\leq (\mu_{\delta_6} \times \nu_{\delta_6}) \left(\bigcup_{i=1}^C \bigcup_{j=1}^C (B^X(x_i,|U|) \times B^Y(y_j,|U|))\right) \\
 &\leq \sum_{i=1}^C \sum_{j=1}^C (\mu_{\delta_6} \times \nu_{\delta_6}) (B^X(x_i,|U|) \times B^Y(y_j,|U|)) \\
 &= C^2 c_E c_F |U|^{t_1+t_2}. 
 \end{align*}
Therefore by the mass distribution principle Lemma~\ref{massdistprinc}~(i), $\overline{\dim}^\Phi (E \times F) \geq t_1 + t_2$, as required. 

(ii) The proof of~(ii) is a straightforward modification of the proof of~(i). 

(iii) The upper bound is just the upper bound of~(i) with $E=F$; the improved bound is the lower bound. Assume $\overline{\dim}^\Phi F > 0$ and let $t \in (0,\overline{\dim}^\Phi F)$. By Lemma~\ref{frostman}~(i) there exists $c_F \in (0,\infty)$ such that for all $\delta_0>0$ there exists $\delta \in (0,\delta_0)$ and a Borel probability measure $\mu_\delta$ with $\mathrm{supp}(\mu_{\delta}) \subseteq F$ such that if $x \in X$ and $\Phi(\delta) \leq r \leq \delta$ then $\mu_{\delta}(B^X(x,r)) \leq c_F r^t$. Then $\mathrm{supp}(\mu_\delta \times \mu_\delta) \subseteq F \times F$, and as in the proof of the lower bound of~(i), if $\Phi(\delta) \leq |U| \leq \delta$ then $(\mu_\delta \times \mu_\delta)(U) \leq C^2 c_F^2 |U|^{2t}$. Therefore by Lemma~\ref{massdistprinc}, $\overline{\dim}^\Phi (F\times F) \geq 2t$, as required. 
\end{proof} 

In the particular case $\Phi(\delta) \coloneqq \frac{\delta}{-\log \delta}$, Proposition~\ref{whenequalsbox} gives $\overline{\dim}^\Phi G = \overline{\dim}_\mathrm{B} G$ and $\underline{\dim}^\Phi G = \underline{\dim}_\mathrm{B} G$ for a subset $G$ of an underlying space $X$. Therefore from~(i) and (ii) we recover the inequalities for the upper and lower box dimensions of product sets in~\cite[Theorem~2.4]{Robinson2013} (which is proven directly, without putting measures on the sets). 
Note also that bounds on the dimensions of products of more than two sets can be obtained by applying Theorem~\ref{producttheorem} iteratively, for example 
\[\overline{\dim}^\Phi (E \times F \times G) \geq \overline{\dim}^\Phi (E \times F) + \underline{\dim}^\Phi G \geq \overline{\dim}^\Phi E +  \underline{\dim}^\Phi F + \underline{\dim}^\Phi G.\]

\subsection{Finite stability}

Our next application of the mass distribution principle is Proposition~\ref{finitestability}, which illustrates an important difference between the upper and lower versions of the dimensions. It was stated in~\cite[Section~14.2.1 2.]{Falconer2021-2} that in Euclidean space, the upper intermediate dimensions are finitely stable but the lower intermediate dimensions are not. 

\begin{prop}\label{finitestability}
Let $\Phi$ be any admissible function. 
\begin{enumerate}[label=(\roman*)]
\item The dimension $\overline{\dim}^\Phi$ is finitely stable: we always have 
\[ \overline{\dim}^\Phi(E \cup F) = \max\{\overline{\dim}^\Phi E, \overline{\dim}^\Phi F\}.\] 

\item The dimension $\underline{\dim}^\Phi$ is \emph{not} finitely stable: there exist compact sets $E,F \subset \mathbb{R}$ such that 
\[\underline{\dim}^\Phi(E \cup F) > \max\{\underline{\dim}^\Phi E, \underline{\dim}^\Phi F\}.\]
\end{enumerate}
\end{prop}

\begin{proof}

It is a straightforward exercise to prove~(i) directly from Definition~\ref{maindefinition}, so we prove only~(ii). 
To do so, we take inspiration from~\cite[Exercises~2.8, 2.9]{Falconer2014}. %
The idea is to construct generalised Cantor sets $E$ and $F$, each of which looks `large' on most scales but `small' on some sequence of scales. We do this in such a way that the sequences of scales where the two sets look small do not even approximately coincide, so for each small $\delta$, either $E$ looks large at \emph{every} scale between $\delta$ and $\Phi(\delta)$, or $F$ looks large at \emph{every} scale between $\delta$ and $\Phi(\delta)$. 

We assume without loss of generality that $\Phi \colon (0,1] \to \mathbb{R}$. 
We inductively define the numbers $k_n \in \{0,1,2,\dotsc \}$ and $e_{10^{k_n}},f_{10^{k_n}}>0$, for $n = 0,1,2,\dotsc$, as follows. Let $k_0\coloneqq 0$, $e_{10^{k_0}}=f_{10^{k_0}}=1$. Having defined $k_n,e_{10^{k_n}},f_{10^{k_n}}$ for some $n = 0,1,2,\dotsc$, there are two cases depending on the parity of $n$. 
If $n$ is even, let $k_{n+1}$ be the smallest integer such that $k_{n+1} > k_n$ and 
\begin{equation}\label{finitestableeqne} (1/3)^{10^{k_{n+1}}-10^{k_n}}f_{10^{k_n}} < \Phi\left((1/5)^{10^{k_n + 1} - 10^{k_n}}(1/3)^{10^{k_n+2}-10^{k_n+1}}e_{10^{k_n}}\right),
\end{equation} 
and let 
\begin{align*} 
&e_{10^{k_{n+1}}} \coloneqq (1/5)^{10^{k_n + 1} - 10^{k_n}}(1/3)^{10^{k_{n+1}}-10^{k_n+1}}e_{10^{k_n}}, \\  &f_{10^{k_{n+1}}} \coloneqq (1/3)^{10^{k_{n+1}}-10^{k_n}}f_{10^{k_n}}. 
\end{align*}
If, on the other hand, $n$ is odd, then let $k_{n+1} > k_n$ be the smallest integer such that 
\begin{equation}\label{finitestableeqnf} (1/3)^{10^{k_{n+1}}-10^{k_n}}e_{10^{k_n}} < \Phi\left((1/5)^{10^{k_n + 1} - 10^{k_n}}(1/3)^{10^{k_n+2}-10^{k_n+1}}f_{10^{k_n}}\right),
\end{equation}
and let 
 \begin{align*} 
 &f_{10^{k_{n+1}}} \coloneqq (1/5)^{10^{k_n + 1} - 10^{k_n}}(1/3)^{10^{k_{n+1}}-10^{k_n+1}}f_{10^{k_n}}, \\
 &e_{10^{k_{n+1}}} \coloneqq (1/3)^{10^{k_{n+1}}-10^{k_n}}e_{10^{k_n}}.
 \end{align*}

Now let $E_1 \coloneqq [0,1]$ and for $j \in \mathbb{N}$, if $10^{k_n} < j \leq 10^{k_n+1}$ for some even $n \in \{0,2,4,\dotsc\}$ then obtain $E_j$ by removing the middle 3/5 of each interval in $E_{j-1}$, otherwise obtain $E_j$ by removing the middle 1/3 of each interval in $E_{j-1}$. 
Let $F_1 \coloneqq [2,3]$ and for $j \in \mathbb{N}$, if $10^{k_n} < j \leq 10^{k_n+1}$ for some odd $n \in \{1,3,5,\dotsc\}$ then obtain $F_j$ from removing the middle 3/5 of each interval in $F_{j-1}$, otherwise obtain $F_j$ by removing the middle 1/3 of each interval in $F_{j-1}$. 
Define $E \coloneqq \bigcap_{j=1}^\infty E_j$ and $F \coloneqq \bigcap_{j=1}^\infty F_j$, noting that both are non-empty and compact subsets of $\mathbb{R}$. 
For all $j \in \mathbb{N}$, let $e_j$ and $f_j$ be the lengths of each of the $2^j$ intervals in $E_j$ and $F_j$ respectively, noting that for each $n \in \mathbb{N}$, the two different definitions that we have given for $e_{10^{k_n}}$ and $f_{10^{k_n}}$ agree by induction. The sequences $e_j$ and $f_j$ lie in $(0,1]$ by induction and converge monotonically to 0. 

We now find an upper bound for $\underline{\dim}_\mathrm{B} E$. Let $n \in \mathbb{N}$ be even. Then $E_{10^{k_n+1}}$ is made up of $2^{10^{k_n+1}}$ intervals, each of length $e_{10^{k_n+1}} = (1/5)^{10^{k_n+1}-10^{k_n}}e_{10^{k_n}} \leq (1/5)^{10^{k_n+1}-10^{k_n}}$. Covering $E$ with these intervals, we see that for all $n \in \mathbb{N}$,  
\begin{equation*} 
\frac{\log N_{e_{10^{k_n+1}}} F (E)}{-\log (e_{10^{k_n+1}})} \leq \frac{\log 2^{10^{k_n+1}}}{\log 5^{10^{k_n+1}-10^{k_n}}} = \frac{10\log 2}{9\log 5}.
 \end{equation*}
Therefore $\underline{\dim}_\mathrm{B} E \leq \frac{10\log 2}{9\log 5}$, and similarly using $F_{10^{k_n+1}}$ for $n$ odd to cover $F$ gives $\underline{\dim}_\mathrm{B} F \leq \frac{10\log 2}{9\log 5}$. Therefore 
\begin{equation}\label{stabilityboxbound}
\frac{10\log 2}{9\log 5} \geq \max\{\underline{\dim}_\mathrm{B} E, \underline{\dim}_\mathrm{B} F\}. 
\end{equation}

To bound $\underline{\dim}^\Phi (E \cup F)$ from below, we use the mass distribution principle. 
Define the sequence $(r_n)_{n \geq 0}$ by 
\[r_n \coloneqq \begin{cases} e_{10^{k_n+2}} = (1/5)^{10^{k_n + 1} - 10^{k_n}}(1/3)^{10^{k_n+2}-10^{k_n+1}}e_{10^{k_n}} & \mbox{if } n \mbox{ even,}\\*
f_{10^{k_n+2}} = (1/5)^{10^{k_n + 1} - 10^{k_n}}(1/3)^{10^{k_n+2}-10^{k_n+1}}f_{10^{k_n}} & \mbox{if } n \mbox{ odd}.
\end{cases}\]
This sequence is strictly decreasing, because if $n \geq 0$ is even then by~\eqref{finitestableeqne}, 
\[ r_{n+1} = f_{10^{k_{n+1}+2}} < f_{10^{k_{n+1}}} < \Phi(e_{10^{k_n+2}}) \leq e_{10^{k_n+2}} = r_n,\]
and similarly if $n$ is odd then $r_{n+1} < r_n$ by~\eqref{finitestableeqnf}. 
Let $\delta \in (0,r_0)$. Define $n_\delta \in \mathbb{N}$ by $r_{n_\delta}\leq \delta < r_{n_\delta - 1}$. 

There are two cases depending on the parity of $n_\delta$. If $n_\delta$ is even, then let $\mu_\delta$ be any Borel probability measure on $F$ which gives mass $2^{-10^{k_{n_\delta + 1}}}$ to each of the $2^{10^{k_{n_\delta + 1}}}$ intervals in $F_{10^{k_{n_\delta + 1}}}$. Let $U$ be a Borel subset of $\mathbb{R}$ with $\Phi(\delta) \leq |U| \leq \delta$. Define $j \in \mathbb{N}$ (depending on $|U|$) by $f_j \leq |U| < f_{j-1}$. By~\eqref{finitestableeqne},  
\[f_{10^{k_{n_\delta+1}}} \leq \Phi(e_{10^{k_{n_\delta}+2}}) = \Phi(r_{n_\delta}) \leq \Phi(\delta) \leq |U| < f_{j-1}, \]
so $j-1 < 10^{k_{n_\delta+1}}$. Also, $f_j \leq |U| \leq \delta < r_{n_\delta - 1} = f_{10^{k_{n_\delta - 1}+2}}$, so in fact $10^{k_{n_\delta-1}+2} < j \leq 10^{k_{n_\delta+1}}$. 
Therefore by the construction of $F$,  
\begin{equation*}%
f_j \geq \left(\frac{1}{5}\right)^{10^{k_{n_\delta-1}+1}}\left(\frac{1}{3}\right)^{j-10^{k_{n_\delta-1}+1}} > \left(\frac{1}{5}\right)^{j/2}\left(\frac{1}{3}\right)^{j/2}.
\end{equation*}
  Since $U$ has diameter less than $f_{j-1}$, it can intersect at most two of the $2^{j-1}$ intervals in $F_{j-1}$. Therefore $U$ can intersect at most $2(2^{10^{k_{n_\delta + 1}} - j})$ of the $2^{10^{k_{n_\delta + 1}}}$ intervals in $F_{10^{k_{n_\delta + 1}}}$. 
  Therefore 
  \begin{equation*}
   \mu_\delta (U) \leq 2(2^{10^{k_{n_\delta + 1}} - j}) (2^{-10^{k_{n_\delta + 1}}}) = 2\left(\left(\frac{1}{3}\right)^{j/2}\left(\frac{1}{5}\right)^{j/2}\right)^{\frac{2\log 2}{\log 15}} \leq 2f_j^{\frac{2\log 2}{\log 15}} \leq 2|U|^{\frac{2\log 2}{\log 15}}. 
  \end{equation*}
  
  If, on the other hand, $n_\delta$ is odd, then let $\mu_\delta$ be a Borel probability measure on $E$ which gives mass $2^{-10^{k_{n_\delta + 1}}}$ to each of the $2^{10^{k_{n_\delta + 1}}}$ intervals in $E_{10^{k_{n_\delta + 1}}}$. As above, if $\Phi(\delta) \leq |U| \leq \delta$ then $\mu_\delta (U) \leq 2|U|^{\frac{2\log 2}{\log 15}}$. Therefore by the mass distribution principle Lemma~\ref{massdistprinc}~(ii) and Proposition~\ref{basicbounds} and~\eqref{stabilityboxbound},  
  \[\underline{\dim}^\Phi(E \cup F) \geq \frac{2\log 2}{\log 15} > \frac{10\log 2}{9\log 5} \geq \max\{\underline{\dim}_\mathrm{B} E, \underline{\dim}_\mathrm{B} F\} \geq \max\{\underline{\dim}^\Phi E, \underline{\dim}^\Phi F\}. \qedhere \]
\end{proof}
It follows from Propositions~\ref{finitestability} and~\ref{basicbounds} and the fact that the Hausdorff dimension of every countable set is $0$ that for any two admissible functions $\Phi_1$ and $\Phi_2$, the three notions of dimension $\dim_\mathrm{H}$, $\underline{\dim}^{\Phi_1}$ and $\overline{\dim}^{\Phi_2}$ are pairwise-distinct, even just working in $\mathbb{R}$. 

Letting $E,F$ be as in Proposition~\ref{finitestability}, applying the mass distribution principle as in the proof of that result at the scales $\delta \coloneqq f_{10^{k_n+2}}$ shows that 
\[ \dim_\mathrm{H} F \leq \underline{\dim}^\Phi F \leq \frac{10\log 2}{9\log 5} < \frac{2\log 2}{\log 15} \leq \overline{\dim}^\Phi F \leq \overline{\dim}_\mathrm{B} F,\]
 and similarly for $E$. 
 Suppose $F$ is the set corresponding to a $\Phi$ satisfying $\frac{\log \delta}{\log \Phi(\delta)} \to 0$ as $\delta \to 0^+$ (for example $\Phi(\delta) = e^{-\delta^{-0.5}}$). Then by Proposition~\ref{compareintermediate}, 
 \[ \dim_\mathrm{H} F < \frac{2\log 2}{\log 15} \leq \overline{\dim}^\Phi F \leq \overline{\dim}_\theta F\]
  for all $\theta \in (0,1]$, so $\overline{\dim}_\theta F$ is discontinuous at $\theta = 0$. Let $\Phi_1$ be an admissible function such that $\Phi_1(f_{10^{k_n+1}}) \leq \Phi_1(f_{10^{k_{n+2}+1}})$ for all sufficiently large $n$. Then for all sufficiently small $\delta$, there exists an odd integer $n(\delta)$ such that $\Phi(\delta) \leq f_{10^{k_{n(\delta)}+1}} \leq \delta$, and the natural cover of $F_{10^{k_{n(\delta)}+1}}$ with $2^{10^{k_{n(\delta)}+1}}$ intervals gives $\overline{\dim}^{\Phi_1} F \leq \frac{10\log 2}{9\log 5} < \frac{2\log 2}{\log 15}$. 
  This gives an indication of how one might construct the admissible functions from Theorem~\ref{recoverinterpolation} below which recover the interpolation for this particular set. 
 
\section{Recovering the interpolation}\label{recoversection} 

It is clear from~\cite[Proposition~2.4]{Falconer2020}, Corollary~\ref{ffkgeneralbounds} and the proof of Proposition~\ref{finitestability} that there are many compact sets with intermediate dimensions discontinuous at $\theta = 0$. For these sets the intermediate dimensions do not fully interpolate between the Hausdorff and box dimensions. 
The main result of this section, Theorem~\ref{recoverinterpolation}, shows that for every compact set there is indeed a family of functions $\Phi$ for which the $\Phi$-intermediate dimensions interpolate all the way between the Hausdorff and box dimensions of the set. %
Moreover, there exists a family of $\Phi$ which interpolates for both the upper and lower versions of the dimensions, and forms a chain in the partial order introduced in Section~\ref{phiintermediatecty}. 
To the best of our knowledge, it is an open problem whether the Assouad-like dimensions studied in~\cite{Garcia2020} fully interpolate between the quasi-Assouad and Assouad dimensions of all sets. 

\begin{thm}\label{recoverinterpolation}
For every non-empty, compact subset $F$, there exists a family 
\[ \{\Phi_s\}_{s \in [\dim_\mathrm{H} F, \overline{\dim}_\mathrm{B} F]}\] 
of admissible functions such that if $s,t$ are such that $\dim_\mathrm{H} F \leq s \leq t \leq \overline{\dim}_\mathrm{B} F$ then the following three conditions hold: 
  \begin{enumerate}[label=(\roman*)]
\item $\overline{\dim}^{\Phi_s} F = s$;
\item $\underline{\dim}^{\Phi_s} F = \min\{s,\underline{\dim}_\mathrm{B} F\}$;
\item $\Phi_s \preceq \Phi_t$. 
\end{enumerate}
\end{thm}
The key definition in the proof is~\eqref{definephis}. The assumption of compactness allows us to take a \emph{finite} subcover in Definition~\ref{hausdorffdef} of Hausdorff dimension, which ensures that $\Phi_s(\delta)$ is well-defined and positive.

\begin{proof} 
Define $\Phi_{\overline{\dim}_\mathrm{B} F}(\delta) \coloneqq \frac{\delta}{-\log \delta}$, so~(i) and (ii) are satisfied for $s = \overline{\dim}_\mathrm{B} F$ by Proposition~\ref{whenequalsbox}. 
We henceforth assume that $\dim_\mathrm{H} F < \overline{\dim}_\mathrm{B} F$, or else there is nothing more to prove.  
The same symbols may take different values in the proofs of parts (i), (ii), (iii). 

(i) %
Let $\Delta \in (0,1/5)$ be such that $0<\frac{\delta}{-\log \delta} < c\delta /3$ for all $\delta \in (0,\Delta)$. %
For now, let $s \in (\dim_\mathrm{H} F, \overline{\dim}_\mathrm{B} F)$. %
For each $\delta \in (0,\Delta)$ there exists a countable cover $\{V_i\}_{i \geq 1}$ of $F$ such that $|V_i| \leq \delta$ for all $i$, and $\sum_i|V_i|^s \leq 2^{-1-2s}$. We may assume that each $V_i$ is non-empty and fix $p_i \in V_i$. Each $V_i \subseteq B(p_i,\max\{2|V_i|,2^{-1-2i/s})\})$, so $\{B(p_i,\max\{2|V_i|,2^{-1-2i/s}\})\}_{i \geq 1}$ is an open cover for $F$. %
Since $F$ is compact, there is a finite subset 
\[ \{U_i\} \subseteq \{B(p_i,\max\{2|V_i|,2^{-1-2i/s}\})\} \] which also covers $F$. %
Now, 
\begin{align*}
\sum_i |U_i|^s \leq \sum_{i \geq 1} |B(p_i,\max\{2|V_i|,2^{-1-2i/s}\})|^s  &\leq \sum_{i=1}^\infty (2^{-2i/s})^s + \sum_{i \geq 1} (4|V_i|)^s \\
&= 1/3 + 4^s \sum_i |V_i|^s \\*
&< 1.
\end{align*} 
Since $\{U_i\}$ is a finite collection of sets, and each has positive diameter as $X$ is uniformly perfect, it follows that $\min_i |U_i| >0$. Therefore $\Phi_s \colon (0,\Delta) \to \mathbb{R}$ is positive and well-defined by  
\begin{equation}\label{definephis}
\begin{aligned}
\Phi_s(\delta) \coloneqq \sup \{ \,x \in [0,\delta/(-\log \delta)] : &\mbox{ there exists a finite cover } \{U_i\} \mbox{ of } F \\* 
&\mbox{ such that } x \leq |U_i| \leq \delta \mbox{ for all } i \mbox{ and } \sum_i|U_i|^s \leq 1 \, \}.
\end{aligned}
\end{equation}
 By construction, $\Phi_s(\delta)/\delta \leq \left(\frac{\delta}{-\log \delta}\right)/\delta \to 0$ as $\delta \to 0^+$, and $\Phi_s$ is increasing in $\delta$, so $\Phi_s$ is admissible. 
  
 We now show that $\overline{\dim}^{\Phi_s} F \leq s$. Given $\eta,\epsilon >0$, define $\delta_0 \coloneqq \min\{ \epsilon^{1/\eta}c^{s/\eta}4^{-s/\eta}, \Delta\}$. Then for all $\delta \in (0, \delta_0)$ there exists a finite cover $\{W_i\}$ of $F$ satisfying $\Phi_s(\delta)/2 \leq |W_i| \leq \delta$ for all $i$, and %
 $\sum_i|W_i|^s \leq 1$. If $|W_i| \geq \Phi_s(\delta)$ then leave $W_i$ in the cover unchanged. If $|W_i| < \Phi_s(\delta)$ then pick any $w_i \in W_i$ and $q_i \in X$ such that $\Phi_s(\delta) \leq d(q_i,w_i) \leq \Phi_s(\delta)/c$. 
 Replace $W_i$ in the cover by $W_i \cup \{q_i\}$. Call the new cover $\{Y_i\}$. By the triangle inequality, 
 \begin{equation*} 
 \Phi_s(\delta) \leq d(q_i,w_i) \leq |W_i \cup \{q_i\}| < \Phi_s(\delta) + \Phi_s(\delta)/c \leq 2\delta/(-c\log \delta)\leq \delta.
 \end{equation*}
 Also 
 \[|W_i \cup \{q_i\}| \leq 2\Phi_s(\delta)/c \leq (4/c)\Phi_s(\delta)/2 \leq 4|W_i|/c. \]
 Therefore 
 \[\sum_i|Y_i|^{s+\eta} \leq \sum_i|Y_i|^s \delta^\eta \leq \delta_0^\eta (4/c)^s \sum_i|W_i|^s \leq \epsilon.\]
  It follows that $\overline{\dim}^{\Phi_s} F \leq s+\eta$, so in fact $\overline{\dim}^{\Phi_s} F \leq s$.
 
 To prove the reverse inequality, assume for a contradiction that $\overline{\dim}^{\Phi_s} F < s$. Then there exists $\delta_1 \in (0,\Delta)$ such that for all $\delta_2 \in (0,\delta_1)$ there exists a cover $\{Z_i\}$ of $F$ such that $\Phi_s(\delta_2) \leq |Z_i| \leq \delta_2$ for all $i$, and $\sum_i |Z_i|^s \leq 3^{-s}c^s$. 
 By Proposition~\ref{whenequalsbox} there exists $\delta_2 \in (0,\delta_1)$ such that $\Phi_s(\delta_2) < \delta_2/(-\log \delta_2)$, and let $\{Z_i\}$ be the cover corresponding to this $\delta_2$, as above. Choose any $z_i \in Z_i$ and let $x_i \in X$ be such that $2|Z_i| \leq d(z_i,x_i) \leq 2|Z_i|/c$. Then by the triangle inequality,  
 \begin{equation*}
 2\Phi_s(\delta_2) \leq 2|Z_i| \leq d(z_i,x_i) \leq |Z_i \cup \{x_i\}| \leq |Z_i| + 2|Z_i|/c \leq (3/c)\delta_2/(-\log \delta_2) <\delta_2. 
 \end{equation*}
 Moreover, $\{Z_i \cup \{x_i\}\}_i$ covers $F$, and 
 \[\sum_i|Z_i \cup \{x_i\}|^s \leq \sum_i (3|Z_i|/c)^s \leq 3^s c^{-s} \sum_i |Z_i|^s \leq 1.\] 
 Therefore 
 \[\Phi_s(\delta_2) \geq \min\{2\Phi_s(\delta_2),\delta_2/(-\log \delta_2)\} > \Phi_s(\delta_2),\]
  a contradiction. Hence $\overline{\dim}^{\Phi_s} F \geq s$ for all $s \in (\dim_\mathrm{H} F, \overline{\dim}_\mathrm{B} F)$, so $\overline{\dim}^{\Phi_s} F = s$. 
 
 Now consider the case $s=\dim_\mathrm{H} F$. Let $N \in \mathbb{N}$ satisfy 
 \[ N > \max\left\{\frac{1}{\overline{\dim}_\mathrm{B} F - \dim_\mathrm{H} F},\frac{1}{\Delta}\right\}. \]
  For $\delta \in (0,1/N]$, let $n \geq N$ be such that $\delta \in (\frac{1}{n+1}, \frac{1}{n}]$, and define 
  \[ \Phi_s(\delta) \coloneqq \min\{\Phi_{s+1/N} (\delta), \dotsc, \Phi_{s+1/n}(\delta)\}.\] 
  Then $\Phi_s(\delta) \leq \Phi_{s+1/N}(\delta) \leq \delta/(-\log \delta)$ for all $\delta \in (0,1/N]$, so $\Phi_s(\delta)/\delta \to 0$ as $\delta \to 0^+$. For all $n \geq N$ and $\delta \in (0,\Delta)$ we have $\Phi_{s+1/n}(\delta)>0$, so if $\delta>0$ then $\Phi_s(\delta)>0$. Moreover, if $\delta_1 \leq \delta_2$, say 
 $\delta_1 \in (\frac{1}{n+1}, \frac{1}{n}]$ and $\delta_2 \in (\frac{1}{m+1}, \frac{1}{m}]$ where $n\geq m \geq N$, then
 \begin{equation*}
 \Phi_s(\delta_1) \leq \min\{\Phi_{s+1/N} (\delta_1), \dotsc, \Phi_{s+1/m}(\delta_1)\} \leq \Phi_s(\delta_2)
 \end{equation*}
  by the monotonicity of each $\Phi_{s+1/i}$. Thus $\Phi_s$ is monotonic, so admissible. 
  For all $n \geq N$ and $\delta \in (0,1/n)$, clearly $\Phi_s(\delta) \leq \Phi_{s+1/n}(\delta)$. 
  Therefore by Proposition~\ref{basicbounds} and Corollary~\ref{hardcomparisoncor}~(i),  
  \[ s = \dim_\mathrm{H} F \leq \underline{\dim}^{\Phi_s} F \leq \overline{\dim}^{\Phi_s} F \leq \overline{\dim}^{\Phi_{s+1/n}} F = s+\frac{1}{n}.\]
   Letting $n \to \infty$ gives $\underline{\dim}^{\Phi_s} F = \overline{\dim}^{\Phi_s} F = s = \dim_\mathrm{H} F$, as required. 
   
   (iii) %
   By construction, (iii) holds since if $\dim_\mathrm{H} F \leq s \leq t \leq \overline{\dim}_\mathrm{B} F$ then $\Phi_s(\delta) \leq \Phi_t(\delta)$ for all sufficiently small $\delta$. %

    (ii) 
    It suffices to prove $\underline{\dim}^{\Phi_s} F \geq \min\{s,\underline{\dim}_\mathrm{B} F\}$, since the opposite inequality follows from Proposition~\ref{basicbounds} and~(i). 
    If $s = \dim_\mathrm{H} F$ or $s = \overline{\dim}_\mathrm{B} F$ then we are done by Propositions~\ref{basicbounds} and~\ref{whenequalsbox}.      
    Suppose $s \in (\dim_\mathrm{H} F, \underline{\dim}_\mathrm{B} F] \cap (\dim_\mathrm{H} F, \overline{\dim}_\mathrm{B} F)$. %
     Assume for a contradiction that $\underline{\dim}^{\Phi_s} F < s$. Let $t,t'$ be such that $\underline{\dim}^{\Phi_s} F < t < t' < s$. Since $t' < \underline{\dim}_\mathrm{B} F$, there exists $\Delta \in (0,\min\{1,|X|\})$ such that $N_\delta (F) \geq \delta^{-t'}$ for all $\delta \in (0,\Delta)$. Reducing $\Delta$ if necessary, we may assume further that $\frac{\delta^{t-t'}}{(-\log \delta)^t} > (1+2/c)^{-s}$ and $-\log \delta \geq 2(1+2/c)$ for all $\delta \in (0,\Delta)$. 
    Since $t > \underline{\dim}^{\Phi_s} F$, for all $\delta_0 > 0$ there exists $\delta \in (0,\min\{\Delta,\delta_0\})$ and a cover $\{U_i\}$ such that $\Phi_s(\delta) \leq |U_i| \leq \delta$ for all $i$, and 
    \begin{equation}\label{longlowersum}
    (1+2/c)^{-s} \geq \sum_i |U_i|^t \geq \sum_i |U_i|^s.
    \end{equation}
     But   
     \[ (1+2/c)^{-s} < \frac{\delta^{t-t'}}{(-\log \delta)^t} = \delta^{-t'}\left(\frac{\delta}{-\log \delta}\right)^t,\]
      so there exists $i$ such that $\delta/(-\log \delta) > |U_i| \geq \Phi_s(\delta)$. %
      If $i$ is such that 
      \[ |U_i| \geq \min\left\{2\Phi_s(\delta),\frac{\delta}{-\log \delta}\right\}\] 
      then leave $U_i$ in the cover unchanged. If, however, $i$ is such that $|U_i| <\min\{2\Phi_s(\delta),\delta/(-\log \delta)\}$ then fix $p_i \in U_i$. Fix $q_i \in X$ such that $2\Phi_s(\delta) \leq d(p,q) \leq 2\Phi_s(\delta)/c$, replace $U_i$ in the cover by $U_i \cup \{q_i\}$, and call the new cover $\{V_i\}_i$. In the case $|U_i| <\min\{2\Phi_s(\delta),\delta/(-\log \delta)\}$, 
    \begin{align*}
     2\Phi_s(\delta) \leq d(p_i,q_i) \leq |U_i \cup \{q_i\}| \leq |U_i| + 2\Phi_s(\delta)/c < 2(1+2/c)\Phi_s(\delta) &\leq \frac{2(1+2/c)\delta}{-\log \delta} \\*
     &\leq \delta.
     \end{align*}
    Then $\min\{\delta/(-\log \delta),2\Phi_s(\delta)\} \leq |V_i| \leq \delta$ for each $i$, and 
    \[ \sum_i |V_i|^s \leq \sum_i ((1+2/c)|U_i|)^s = (1+2/c)^s \sum_i|U_i|^s \leq 1, \]
    by~\eqref{longlowersum}. 
    This means that $\Phi_s(\delta) \geq \min\{2\Phi_s(\delta), \delta/(-\log \delta)\} > \Phi_s(\delta)$, a contradiction. Hence $\underline{\dim}^{\Phi_s} F \geq s$ for all $s \in (\dim_\mathrm{H} F, \underline{\dim}_\mathrm{B} F]$. %
    
   Now suppose $s \in (\underline{\dim}_\mathrm{B} F, \overline{\dim}_\mathrm{B} F)$. %
   By (iii), $\Phi_{\underline{\dim}_\mathrm{B} F} \preceq \Phi_s$, so by what we have just proved, $\min\{s,\underline{\dim}_\mathrm{B} F\} = \underline{\dim}_\mathrm{B} F \leq \underline{\dim}^{\Phi_{\underline{\dim}_\mathrm{B} F}} F \leq \underline{\dim}^{\Phi_s} F$. 
        Together, the cases show that for all $s \in [\dim_\mathrm{H} F,\overline{\dim}_\mathrm{B} F]$ we have $\underline{\dim}^{\Phi_s} F \geq \min\{s,\underline{\dim}_\mathrm{B} F\}$ and hence $\underline{\dim}^{\Phi_s} F = \min\{s,\underline{\dim}_\mathrm{B} F\}$, as required.  
\end{proof}

In the definition~\eqref{definephis} of $\Phi_s$, any positive constant would work in place of the constant 1, so there are many different $\Phi_s$ that will work. The family of dimensions $\overline{\dim}^{\Phi_s}$ and $\underline{\dim}^{\Phi_s}$ may not vary continuously for all sets, as shown by the following proposition. 

\begin{prop}\label{interpolatenotcts}
There exist non-empty, compact subsets $F,G$ of $\mathbb{R}$ such that: 

(i) if $(\Phi_s)_{s \in (\dim_\mathrm{H} F, \underline{\dim}_\mathrm{B} F)}$ is any family of admissible functions such that $\overline{\dim}^{\Phi_s} F = s$ for all $s \in (\dim_\mathrm{H} F, \underline{\dim}_\mathrm{B} F)$ then the function $s \mapsto \overline{\dim}^{\Phi_s} G$ is not continuous on $(\dim_\mathrm{H} F, \underline{\dim}_\mathrm{B} F)$, and

(ii) if $(\Psi_s)_{s \in (\dim_\mathrm{H} F, \underline{\dim}_\mathrm{B} F)}$ is such that $\underline{\dim}^{\Psi_s} F = s$ for all $s \in (\dim_\mathrm{H} F, \underline{\dim}_\mathrm{B} F)$ then the function $s \mapsto \underline{\dim}^{\Psi_s} G$ is not continuous on $(\dim_\mathrm{H} F, \underline{\dim}_\mathrm{B} F)$. 
\end{prop}

\begin{proof}

Let $G \coloneqq \{0\} \cup \{1/n : n \in \mathbb{N}\}$, so $\dim_\theta G = \frac{\theta}{1+\theta}$ for all $\theta \in [0,1]$ by~\cite[Proposition~3.1]{Falconer2020}. Let $F = E \cup G$ for any compact countable set $E \subset \mathbb{R}$ with $\underline{\dim}_\mathrm{B} E = \dim_\mathrm{A} E = 1/4$, so as in~\cite[Example~3]{Falconer2020} $\dim_\mathrm{H} F = 0$ and $\dim_\theta F = \max\left\{ \frac{\theta}{1+\theta},1/4\right\}$ for all $\theta \in (0,1]$. 
We now prove (i) using Proposition~\ref{compareintermediate}; the proof of (ii) is similar. Suppose $(\Phi_s)_{s \in (\dim_\mathrm{H} F, \underline{\dim}_\mathrm{B} F)}$ satisfies $\overline{\dim}^{\Phi_s} F = s$ for all $s \in (\dim_\mathrm{H} F, \underline{\dim}_\mathrm{B} F)$. Then if $s>1/4$ then $\overline{\dim}^{\Phi_s} F = s > 1/4 = \dim_{1/3} F$, so by Proposition~\ref{compareintermediate}, 
\[ \limsup_{\delta \to 0^+} \frac{\log \Phi_s(\delta)}{\log \delta} > 1/3\] and $\overline{\dim}^{\Phi_s} G \geq \dim_{1/3} G = 1/4$. For all $s < 1/4$, $\frac{\log \Phi_s(\delta)}{\log \delta} \to 0$ as $\delta \to 0^+$, so since $\dim_\theta G = \frac{\theta}{1+\theta} \to 0$ as $\theta \to 0$, it follows that $\dim^{\Phi_s} G = 0$. Therefore the function $s \mapsto \overline{\dim}^{\Phi_s} G$ is not continuous at $s=1/4$. 
\end{proof}

\section*{Acknowledgements}
\addcontentsline{toc}{section}{Acknowledgements}

This work was completed whilst the author was a PhD student in the School of Mathematics and Statistics at the University of St Andrews. He was financially supported by a Leverhulme Trust Research Project Grant (RPG-2019-034). 
He thanks Kenneth Falconer, István Kolossváry, Lars Olsen, Alex Rutar, P\'eter Varj\'u, Justin Tan, and especially Jonathan Fraser, for helpful discussions and comments. 
He thanks an anonymous referee for helpful comments which improved the exposition of the paper. 

\section*{References}
\addcontentsline{toc}{section}{References}

\printbibliography[heading=none]

\bigskip
\footnotesize
{\parindent0pt
\textsc{Amlan Banaji \\
Department of Mathematical Sciences \\
Loughborough University \\
Loughborough, LE11 3TU \\
United Kingdom \\ 
Email:} \texttt{A.F.Banaji@lboro.ac.uk}\par\nopagebreak
}

\end{document}